\documentclass{amsart}

\usepackage{txfonts}
\usepackage{amsmath,amssymb,amsthm,amsfonts}
\usepackage{graphicx}

\usepackage{color}
\definecolor{Chocolat}{rgb}{0.36, 0.2, 0.09}
\definecolor{BleuTresFonce}{rgb}{0.215, 0.215, 0.36}

\usepackage[colorlinks,final,hyperindex]{hyperref}
\hypersetup{citecolor=BleuTresFonce, linkcolor=Chocolat}

\usepackage[all]{xy}

\DeclareMathAlphabet{\mathbbold}{U}{bbold}{m}{n}

\DeclareSymbolFont{rsfscript}{OMS}{rsfs}{m}{n}
\DeclareSymbolFontAlphabet{\mathrsfs}{rsfscript}

\DeclareFontFamily{OMS}{rsfs}{\skewchar\font'177}
\DeclareFontShape{OMS}{rsfs}{m}{n}{%
      <5> rsfs5
      <6> <7> rsfs7
      <8> <9> <10> rsfs10
      <10.95> <12> <14.4> <17.28> <20.74> <24.88> rsfs10
      }{}

\def\calM{\mathcal{M}}

\newcommand{\ac}{\scriptstyle \text{\rm !`}}
\DeclareMathOperator{\D}{D}

\DeclareMathOperator{\Hom}{Hom}
\DeclareMathOperator{\Diff}{Diff}

\DeclareMathOperator{\Com}{Com}

\DeclareMathOperator{\End}{End}
\DeclareMathOperator{\cEnd}{End}

\DeclareMathOperator{\BV}{BV}

\DeclareMathOperator{\str}{tr}

\makeatletter

\theoremstyle{plain}

\newtheorem {theorem}{Theorem}
\newtheorem {lemma}{Lemma}
\newtheorem {corollary}{Corollary}

\newtheorem {question}{Question}
\newtheorem {proposition}{Proposition}

\theoremstyle{definition}

\newtheorem {definition}{Definition}

\theoremstyle{remark}
\newtheorem{remark}{Remark}[section]

\numberwithin{equation}{section}

\begin{document}

\title[Givental group action on TFTs and homotopy Batalin--Vilkovisky algebras]{Givental group action on Topological Field Theories\\ and homotopy Batalin--Vilkovisky algebras}

\author{Vladimir Dotsenko}
\address{School of Mathematics, Trinity College, Dublin 2, Ireland}
\email{vdots@maths.tcd.ie}

\author{Sergey Shadrin}
\address{Korteweg-de Vries Institute for Mathematics, University of Amsterdam, P. O. Box 94248, 1090 GE Amsterdam, The Netherlands}
\email{s.shadrin@uva.nl}

\author{Bruno Vallette}
\address{Laboratoire J.A.Dieudonn\'e, Universit\'e de Nice Sophia-Antipolis, Parc Valrose, 06108 Nice Cedex 02, France}
\email{brunov@unice.fr}

\begin{abstract}
In this paper, we initiate the study of the Givental group action on Cohomological Field Theories in terms of homotopical algebra. More precisely, we show that the stabilisers of Topological Field Theories in genus $0$ (respectively in genera $0$ and $1$) are in one-to-one correspondence with commutative homotopy Batalin--Vilkovisky algebras (respectively wheeled commutative homotopy BV-algebras). 
\end{abstract}

\maketitle 

\setcounter{tocdepth}{1}

\tableofcontents

\section*{Introduction}\label{sec:intro}

The Deligne--Mumford--Knudsen compactifications $\overline{\mathcal{M}}_{g,n}$ of the moduli spaces of curves with marked points form a modular operad. This algebraic structure is defined via the mappings of moduli spaces of curves that identify two marked points of one or two curves. This modular operad structure passes to homology, and algebras over ${H}_\bullet(\overline{\mathcal{M}}_{g,n})$ are called Cohomological Field Theories, or CohFTs for short. The notion of a CohFT was introduced by Kontsevich and Manin in~\cite{KonMan} in order to capture the main properties of Gromov--Witten invariants of target varieties. Recently a new set of natural examples of CohFTs came from the quantum singularity theory of Fan--Jarvis--Ruan--Witten~\cite{FJR}.

Cohomological field theories also play a crucial role in the formulation of one of the versions of the Mirror Symmetry conjecture. Namely, in \cite{BCOV}, Bershadsky, Cecotti, Ooguri and Vafa introduced a construction of a mirror partner (B-side) for the Gromov--Witten invariants (A-side) of a mirror dual
Calabi-Yau manifold. More precisely, they considered what is now called the BCOV action on the Dolbeault complex of a Calabi--Yau manifold. Barannikov and Kontsevich showed in \cite{BarKon} that the critical value of the BCOV action indeed provides a genus 0 CohFT structure on the Dolbeault cohomology.

\smallskip

In the series of papers \cite{Giv01, Giv01bis, Giv04}, Givental developed a particular group action on a special class of formal power series (the ``$R$-action''). Using this group action and some extra operators, he proposed an explicit conjectural formula for the Gromov--Witten invariants of target varieties with semi-simple quantum cohomology (e.g., projective spaces) in terms of the genus $0$ data and known Gromov--Witten invariants of a finite number of disjoint points. This conjecture was proved by Teleman in~\cite{Teleman} via a complete classification of semi-simple CohFTs.

The (co)homology ring of the Deligne--Mumford--Knudsen moduli space of stable curves is an intricate subject of study. One way to approximate it is to
look at the tautological rings $RH^\bullet(\overline{\mathcal{M}}_{g,n})$ that are defined as the subalgebras of the cohomology algebras of $\overline{\mathcal{M}}_{g,n}$ that contain all natural cohomology classes (like the $\psi$-classes and the $\kappa$-classes). In fact, using the Poincar\'e duality, one can define a structure of a modular operad on the collection of the cohomology algebras $\{H^\bullet(\overline{\mathcal{M}}_{g,n})\}$ of the moduli spaces 
$\overline{\mathcal{M}}_{g,n}$. Then the collection $\{RH^\bullet(\overline{\mathcal{M}}_{g,n})\}$ can be defined as the collection of the minimal system of subalgebras of $H^\bullet(\overline{\mathcal{M}}_{g,n})$ that forms a modular suboperad of $H^\bullet(\overline{\mathcal{M}}_{g,n})$. In genus $0$, the tautological ring coincides with the full cohomology ring: $RH^\bullet(\overline{\mathcal{M}}_{0,n})=H^\bullet(\overline{\mathcal{M}}_{0,n})$.

In \cite{FSZ10}, Faber, the second author and Zvonkine proved that the Givental group acts on representations of the modular operad 
$RH^\bullet(\overline{\mathcal{M}}_{g,n})$ of a given dimension. Later Kazarian~\cite{Kaz07} and Teleman~\cite{Teleman} observed that there is a way to describe the Givental group action as an action on CohFTs of a given dimension, that is, representations of the modular operad $\{H_\bullet(\overline{\mathcal{M}}_{g,n})\}$.

\smallskip

In the recent preprint \cite{DV}, Drummond-Cole and the third author described, in terms of the Homotopy Transfer Theorem, or HTT for short, how the underlying homology groups of some differential graded Batalin--Vilkovisky algebras can be endowed with a natural Frobenius manifold structure. In general, the HTT produces homotopy BV-algebra structures on homology. But, it is proved in \emph{loc.\ cit.}\ that the transferred homotopy BV-algebra gives rise to a Frobenius manifold when the induced operator $\Delta$ and its higher homotopies vanish. (In our definitions, a Frobenius manifold is just a genus $0$ CohFT structure.) This generalises the Barannikov--Kontsevich Frobenius manifold structure and it provides higher homotopical invariants which allow one to reconstruct the homotopy type of the original dg $\BV$-algebra, for instance the Dolbeault complex.

\smallskip

The latter result hints at a certain role played by homotopy BV-algebras in the context of the Mirror Symmetry conjecture and the Givental group action. The present paper initiates the study of Givental group action in terms of homotopy BV-algebras, as follows.

\smallskip

In genus $0$, on the one hand, we restrict ourselves to topological field theories, or TFT for short, which are cohomological field theories concentrated in degree $0$. A genus $0$ TFT is equivalent to a commutative algebra structure. On the other hand, we restrict ourselves to commutative $\BV_\infty$-algebras, which are $\BV_\infty$-algebras where many higher operations vanish, so that only the commutative product and a sequence of differential operators $D_l$ of order at most $l$, $l\ge 1$, remain. Theorem~\ref{comm-bv-infty} states that the data of a commutative $\BV_\infty$-algebra structure on a TFT is equivalent to the data of an element of the Lie algebra of the Givental group which preserves the given TFT. 

In genera $0$ and $1$, on the one hand, we restrict ourselves to genera $0$ and $1$ TFTs. This algebraic structure is equivalent to a commutative algebra equipped with a compatible trace. On the other hand, we restrict ourselves to wheeled commutative $\BV_\infty$-algebras, which are generalizations of the notion of commutative $\BV_\infty$-algebras, but equipped with a coherent trace. Theorem~\ref{wheeled-comm-bv-infty} states that the data of a wheeled commutative $\BV_\infty$-algebra structure on a genera $0$ and $1$ TFT is equivalent to the data of an element of the Givental group which preserves the given TFT. This theorem partly relies on a recent proof~\cite{Pet12} of the Gorenstein conjecture for moduli spaces in genus~$1$~\cite{HL,Pan}.

\smallskip

An algebra over a modular operad has to be equipped with a scalar product and is, therefore, finite dimensional. To get rid of the dimension assumption, we recall that algebras over an operad or over a wheeled operad do not have to be finite dimensional, and consider, respectively, the operad $\{H_\bullet(\overline{\calM}_{0,n})\}$ and the wheeled operad  $\{H_\bullet(\overline{\calM}_{\leq 1,n})\}$ associated to the modular operad of cohomology algebras $\{H_\bullet(\overline{\calM}_{g,n})\}$.  It is worth mentioning that the main result of \cite{DV} relies on the Koszul duality of the two operads $\{H_\bullet(\overline{\mathcal{M}}_{0,n})\}$ and $\{H_\bullet({\mathcal{M}}_{0,n})\}$. The Koszul duality for the pair of wheeled operads $\{H_\bullet(\overline{\mathcal{M}}_{\leq 1,n})\}$ and $\{H_\bullet({\mathcal{M}}_{\leq 1,n})\}$, as well as the Koszul duality for the pair of (wheeled) properads $\{H_\bullet(\overline{\mathcal{M}}_{g,n})\}$ and $\{H_\bullet({\mathcal{M}}_{g,n})\}$ are interesting  open 
questions. In higher genera, the study of the precise relationship between the full action on the Givental group on a general CohFT with homotopy $\BV$-algebras is yet to be finished. This will be the subject of a future work.

\smallskip

The paper is organised as follows. In Section~\ref{sec:notions}, we have accumulated all the necessary background information on the intersection theory for moduli spaces of curves, on operads and the Givental group action. In Section~\ref{sec:DifOp}, we recall the definition of a differential operator on a commutative algebra, introduce a new notion of compatibility with the trace, and prove some auxiliary results on differential operators compatible with traces in the sense of that definition. The first theorem, in genus $0$, is proved in Section~\ref{com-bv-infty-via-givental}. The second theorem, in genera $0$ and $1$, is proved in Section~\ref{sec:wheeled-ho-bv-givental}. The first appendix explains why differential operators arise naturally in the context of wheeled homotopy $\BV$-algebras. The second appendix deals with a generalized BCOV theory.

\subsection*{Conventions. } Throughout the paper, all vector spaces, unless otherwise specified, are defined over the field of complex numbers~$\mathbb{C}$. Most of our constructions implicitly assume that we work with the tensor category of graded vector spaces (that is, ``commutative'' means ``graded commutative'', ``trace'' means ``supertrace'', ``commutator'' means ``supercommutator'' etc.); for the convenience of the reader, we write all the formulae for the elements of degree zero, keeping in mind that in general signs will appear in formulae according to the Koszul sign rule for evaluating operations on elements. The words ``commutative algebra'' always refer to a commutative associative algebra. We use the ``topologist's notation'' for finite sets, putting $\underline{n}:=\{1, \ldots , n\}$ and $[n]:=\{0, 1, \ldots , n\}$. For a given product $a_1a_2\cdots a_n$ of factors indexed by $\underline{n}$, we denote by~$a_I$, for $I\subset\underline{n}$, the product of factors whose subscripts are in~$I$.

\subsection*{Acknowledgements. } The first author was supported by Grant GeoAlgPhys 2011--2013 awarded by the University of Luxembourg. The second author was supported by the Netherlands Organisation for Scientific Research. The third author would like to thank the Max Planck Institute for Mathematics (Bonn) for a long term invitation. The authors would like to thank the University of Luxembourg and the Max Planck Institute for Mathematics (Bonn) for the excellent working conditions enjoyed during their visits there. The authors also wish to thank Maxim Kazarian for sending a copy of his unpublished manuscript~\cite{Kaz07}.

\section{Recollections}\label{sec:notions}

In this section, we recall only the most basic information. For more information on the moduli spaces of curves, see the survey \cite{Vakil}, for a detailed discussion of cohomological field theories, correlators, and relations between them --- the book \cite{Manin}, for information on operads~--- the book \cite{LodayVallette10}.

\subsection{The moduli spaces of curves}

The \emph{moduli space of curves} $\calM_{g,n}$ parametrises smooth complex curves of genus $g$ with $n$ ordered marked points. Under the usual assumption $2g-2+n>0$, it is a smooth complex orbifold of dimension $3g-3+n$. Its Deligne--Mumford--Knudsen compactification $\overline{\calM}_{g,n}$ parametrises stable curves of genus $g$ with $n$ ordered marked points. A \emph{stable curve} is a connected curve with a finite automorphism group whose allowed singularities are simple nodes. The space $\overline{\calM}_{g,n}$ is a smooth compact complex orbifold of (complex) dimension~$3g-3+n$.

Let us recall three kinds of ``natural mappings'' that can be defined  between the different moduli spaces of curves.
First, there are projections $$\pi\colon \overline{\calM}_{g,n+1}\to \overline{\calM}_{g,n}$$ that forget the last marked point. 
Second, the identification of the last two marked points gives rise to the 2-to-1 mapping $$\sigma\colon\overline{\calM}_{g-1,n+2}\to\overline{\calM}_{g,n}$$ whose image is the boundary divisor of irreducible curves with one node. 
Third, gluing together two curves along their last marked points gives rise to the mapping $$\rho\colon \overline{\calM}_{g_1,n_1+1}\times \overline{\calM}_{g_2,n_2+1}\to \overline{\calM}_{g,n}, \quad g_1+g_2=g, n_1+n_2=n.$$  
These mappings $\rho$ produce the other irreducible boundary divisors of the compactification of $\overline{\calM}_{g,n}$.
 
The Deligne--Mumford--Knudsen compactification $\overline{\calM}_{g,n}$ has a natural stratification by the topological type of stable curves. The images of the mappings $\sigma$ and $\rho$ give a complete description of strata in codimension~$1$. An irreducible boundary stratum of codimension $k$ in $\overline{\calM}_{g,n}$ is represented as an image $p(S)$ of the product $S=\overline{\calM}_{g_1,n_1}\times \cdots \times \overline{\calM}_{g_a, n_a}$. Here $p$ is a composition of the $k$ natural mappings ($\sigma$ and/or $\rho$) described above.

The cohomology algebras of the moduli space of curves are complicated objects, and only limited information about them is available. However, a special system of subalgebras, called the tautological rings, is more accessible. The system of \emph{tautological rings} $RH^\bullet(\overline{\calM}_{g,n})\subset H^\bullet(\overline{\calM}_{g,n},\mathbb{C})$ is defined as the minimal system of subalgebras of the aforementioned cohomology algebras that is closed under the push-forwards and the pull-backs via the natural mappings. The cohomology classes in $RH^\bullet(\overline{\calM}_{g,n})$ are called the \emph{tautological classes}. The elements of the tautological ring that will be of crucial importance for our computations are the following ``$\psi$-classes''.

\begin{definition}[$\psi$-classes]
Both the moduli space $\calM_{g,n}$ and its compactification $\overline{\calM}_{g,n}$ have $n$ tautological line bundles $\mathbb{L}_i$. The fibre of $\mathbb{L}_i$ over a point represented by a curve $C_g$ with marked points $x_1,\dots,x_n$ is equal to the cotangent line $T^*_{x_i}C_g$. The cohomology class $\psi_i$ of $\overline{\calM}_{g,n}$ is defined as the first Chern class of the line bundle~$\mathbb{L}_i$:
\begin{equation}
\psi_i=c_1(\mathbb{L}_i)\in H^2(\overline{\calM_{g,n}},\mathbb{Q}).
\end{equation}
\end{definition}

\smallskip

Using the Poincar\'e duality and the push-forward on the homology, one can define the push-forward maps $\pi_*$, $\sigma_*$, and $\rho_*$ on the cohomology as follows: $\pi_*$ is the composite
\begin{equation}
H^\bullet(\overline{\calM}_{g,n+1})\to H_{6g-6+2n+2-\bullet}(\overline{\calM}_{g,n+1})\to H_{6g-6+2n+2-\bullet}(\overline{\calM}_{g,n})\to H^{\bullet-2}(\overline{\calM}_{g,n}), 
\end{equation}
$\sigma_*$ is the composite
\begin{equation}
H^\bullet(\overline{\calM}_{g-1,n+2})\to H_{6g-12+2n+4-\bullet}(\overline{\calM}_{g-1,n+2})\to H_{6g-12+2n+4-\bullet}(\overline{\calM}_{g,n})\to H^{\bullet+2}(\overline{\calM}_{g,n}),  
\end{equation}
and $\rho_*$ is the composite
\begin{multline}
H^\bullet(\overline{\calM}_{g_1,n_1+1})\otimes H^\bullet(\overline{\calM}_{g_2,n_2+1})\to H^\bullet(\overline{\calM}_{g_1,n_1+1}\times\overline{\calM}_{g_2,n_2+1})\to\\ \to 
H_{6g_1+6g_2-12+2n_1+2n_2+4-\bullet}(\overline{\calM}_{g_1,n_1+1}\times\overline{\calM}_{g_2,n_2+1})\to \\ \to
H_{6g_1+6g_2-12+2n_1+2n_2+4-\bullet}(\overline{\calM}_{g_1+g_2,n_1+n_2+1}) \to H^{\bullet+2}(\overline{\calM}_{g_1+g_2,n_1+n_2+1}). 
\end{multline}

\subsection{Operads}\label{subsec:Operad}
The various notions of operads encode the algebraic structures defined by various types of operations. The toy model of operads is the endomorphism operad $\mathrm{End}_V$ whose components $\mathrm{Hom}(V^{\otimes n}, V)$ consist of multilinear maps defined on a given vector space. In particular, $\End_V(1)$ is the space of all linear operators on~$V$; we shall keep the usual notation $\End(V)$ for it, hoping that no confusion would arise.

\begin{definition}[Operad]
An \emph{operad} is an \emph{$\mathbb{S}$-module} $\mathcal P$, that is a collection of right $\mathbb{S}_n$-modules $\lbrace \mathcal{P}(n)\rbrace_{n\in \mathbb{N}}$, endowed with equivariant \emph{partial compositions}
\begin{equation}
\circ_i \ : \ \mathcal{P}(m)\otimes \mathcal{P}(n) \to  \mathcal{P}(m+n-1), \quad \text{for} \ 1\leq i \leq m \ , 
\end{equation}
satisfying 
\begin{displaymath}
\begin{array}{llcll}
&(\mu \circ_i \nu)\circ_{i-1+j}\omega  &=& \mu \circ_i (\nu\circ_{j}\omega ), &\mathrm{for }\ 1\leq i\leq l, 1\leq j\leq m,  \\
&(\mu \circ_i \nu)\circ_{k-1+m}\omega  &=& (\mu \circ_k\omega)\circ_{i}\nu, &\mathrm{for }\  1\leq i < k\leq l, \\
\end{array}
\end{displaymath}
for any $\mu \in \mathcal{P}(l), \nu\in \mathcal{P}(m), \omega\in \mathcal{P}(n)$.
An operad is required to be \emph{unital}, that is equipped with an element $\mathrm{I} \in \mathcal{P}(1)$ acting as a unit for the partial compositions. 
\end{definition}

Elements of an operad model operations with $n$ inputs and one output. The partial composition $\circ_i$ amounts to composing one operation at the $i^{\rm th}$ input of another one, as the next figure shows.

\begin{center}
\includegraphics[scale=0.2]{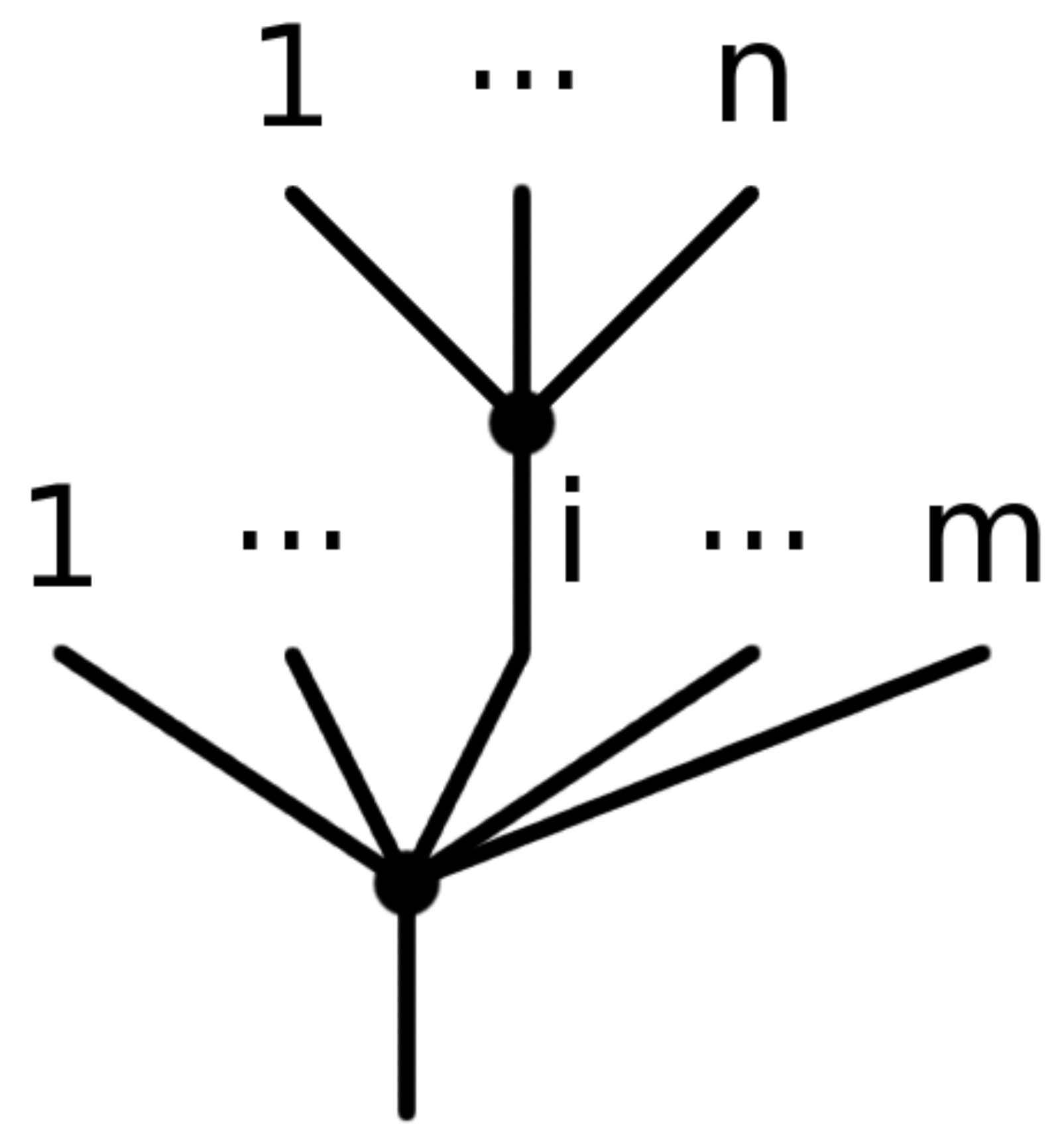} 
\end{center}

The category of $\mathbb{S}$-modules is equivalent to  the category $\mathsf{Vect}^{\mathsf{Bij}^{\textrm{op}}}$ of  contravariant functors from the groupoid  $\mathsf{Bij}$ of finite sets and bijections to the category $\mathsf{Vect}$ of vector spaces:
\begin{equation}
 \begin{array}{lcl}
\mathbb{S}\mathsf{-modules} & \cong & \mathsf{Vect}^{\mathsf{Bij}^{\textrm{op}}} \\
\lbrace \mathcal{P}(n)\rbrace_{n\in \mathbb{N}} & \mapsto &  \mathcal{P}(X):=\Big(\bigoplus_{f:\underline{n}\to X} \mathcal{P}(n)\Big)_{\mathbb{S}_{n}}   \\
\mathcal{P}(\underline{n})&\mathbin{\reflectbox{$\mapsto$}} & \mathcal{P}(-) \ , 
\end{array}
\end{equation}
where in the space of coinvariants on the right-hand side the sum is over all the bijections from  $\underline{n}$ to $X$ and where the right action of $\sigma\in \mathbb{S}_{n}$ on $(f;\mu)$ for $\mu\in \mathcal{P}(n)$ is given by
 $(f;\mu)^{\sigma} :=(f\sigma;\mu^{\sigma}).$ So, from now on, we freely identify them.

\begin{definition}[Modular operad \cite{GetKap98}]
A \emph{modular operad} is 
 a \emph{graded $\mathbb{S}$-module} $\mathcal P$, that is a graded collection of right $\mathbb{S}_{n}$-modules $\{\mathcal{P}_g(n)\}_{g,n\in \mathbb{N}}$, endowed with equivariant \emph{partial compositions}
\begin{equation}
\circ_i^j \ : \ \mathcal{P}_g(m)\otimes \mathcal{P}_{g'}(n) \to  \mathcal{P}_{g+g'}(m+n-2), \quad \text{for} \ 1\leq i \leq m \ \text{and} \ 1\leq j \leq n
\end{equation}    
\begin{center}
\includegraphics[scale=0.2]{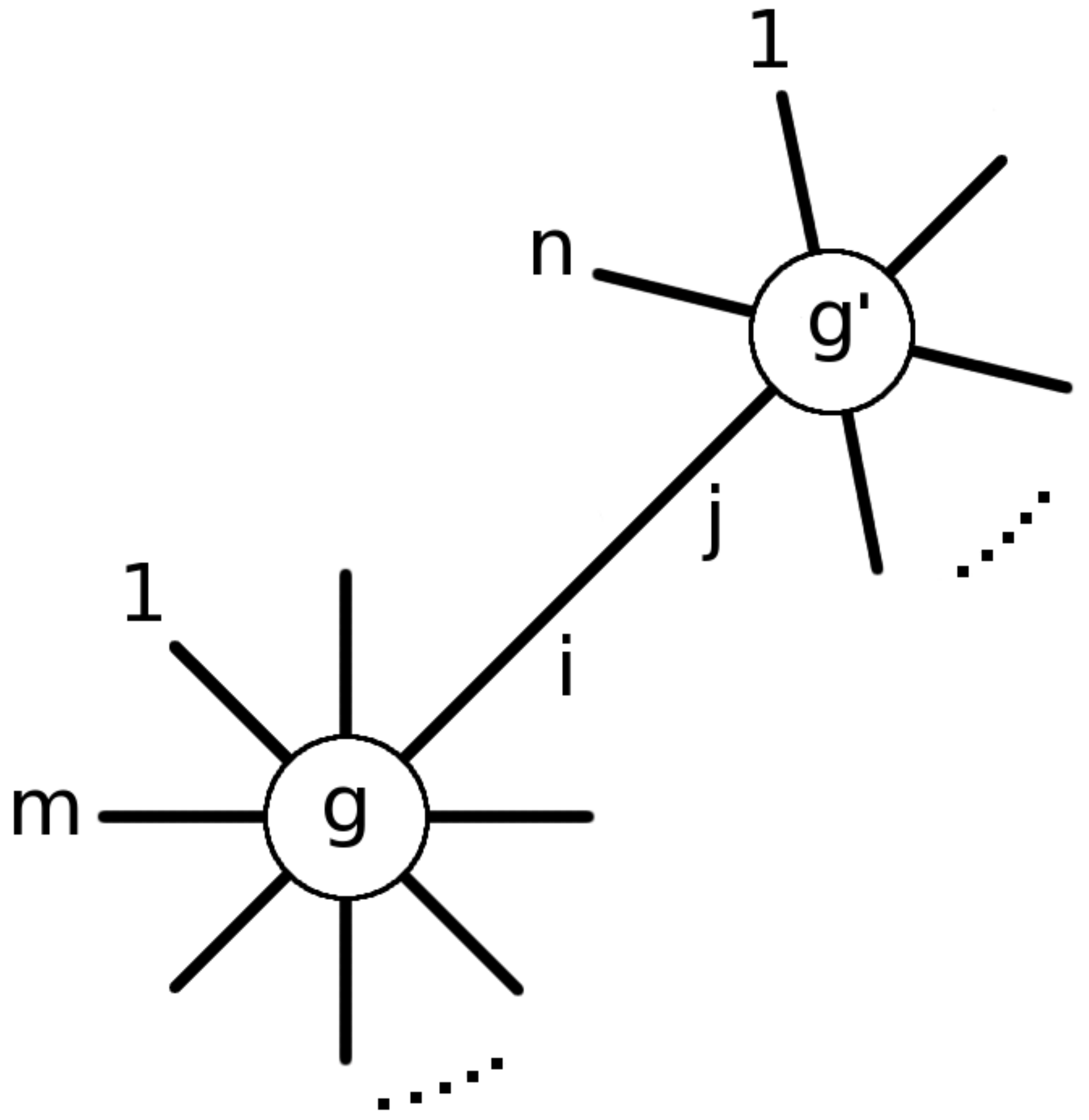} 
\end{center}
and equivariant \emph{contractions}
\begin{equation}
\xi_{ij} \ : \  \mathcal{P}_g(n) \to  \mathcal{P}_{g+1}(n-2), \quad \text{for} \ 1\leq i\neq j \leq n    \ . 
\end{equation}
\begin{center}
\includegraphics[scale=0.2]{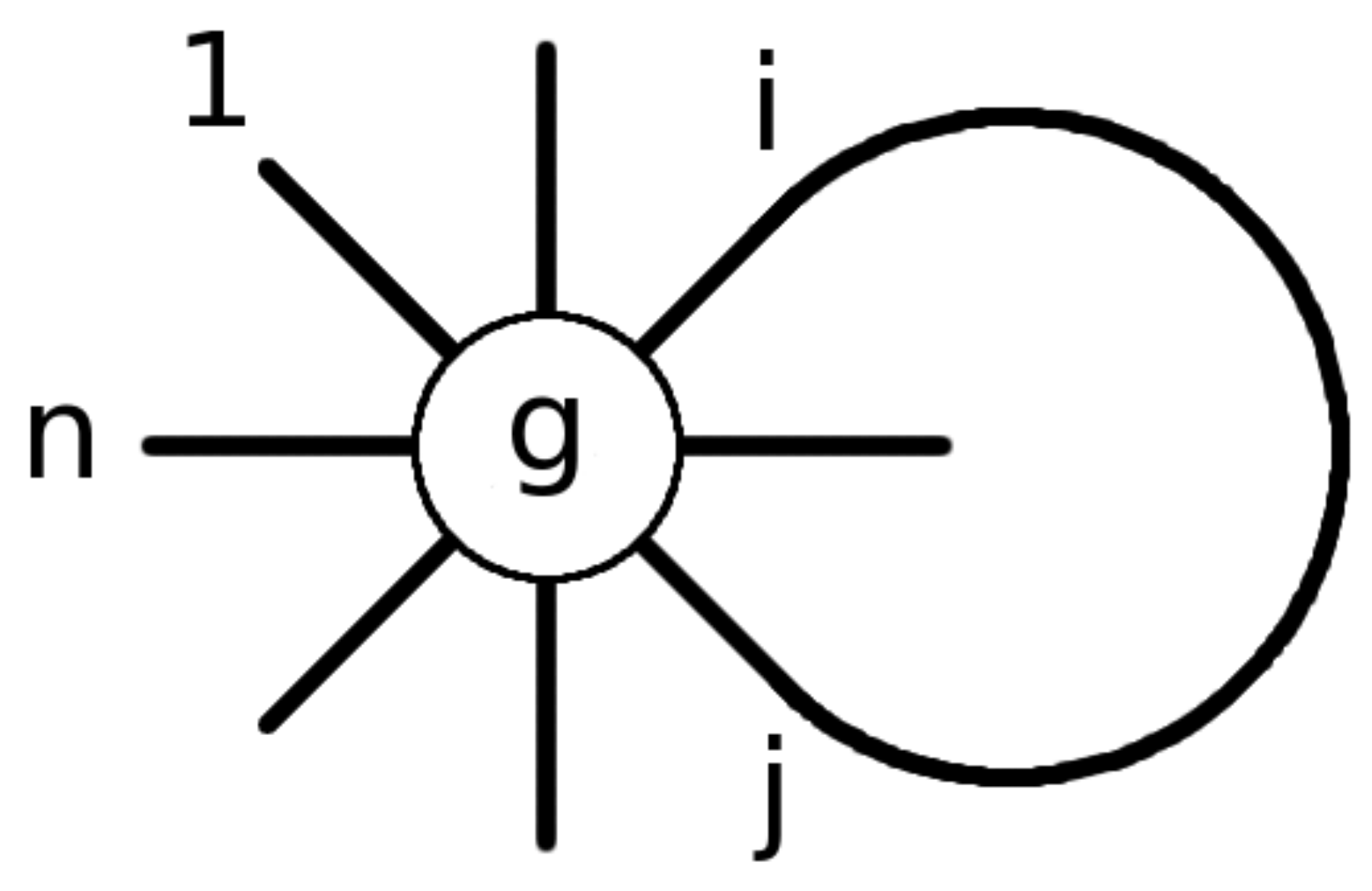} 
\end{center}
These structure maps are required to satisfy the following properties, for every choice of $\mu \in \mathcal{P}_g(X)$, $\nu\in \mathcal{P}_{g'}(Y)$, and $\omega\in \mathcal{P}_{g''}(Z)$:
\begin{displaymath}
(\mu \circ_i^j \nu)\circ_k^l\omega = \left\{ 
\begin{array}{ll}
 \mu \circ_i^j (\nu\circ_k^l\omega ), &\mathrm{when}\ k\in Y,  \\
 (\mu \circ_k^l\omega)\circ_{i}^j\nu , &\mathrm{when}\  k\in X, \\
\end{array}\right.
\end{displaymath}
for any $i\in X$, $j\in Y$, and $l\in Z$;
\begin{equation}
\xi_{ij} \xi_{kl} \, \mu =\xi_{kl} \xi_{ij} \, \mu   
\end{equation}
for any distinct $i$, $j$, $k$, and $l$ in $X$; 
\begin{displaymath}
\xi_{ij}(\mu \circ_k^l \nu) = \left\{ 
\begin{array}{ll}
\xi_{ij}(\mu) \circ_k^l \nu, &\mathrm{when}\ i,j\in X,  \\
\mu \circ_k^l \xi_{ij}(\nu), &\mathrm{when}\ i,j\in Y,  \\
\xi_{kl}(\mu \circ_i^j \nu), &\mathrm{when}\ i\in X \ \text{and}\ j\in Y,  \\
\xi_{kl}(\nu \circ_i^j \mu), &\mathrm{when}\ i\in Y \ \text{and}\ j\in X,  
\end{array}\right.
\end{displaymath}
for any $k\in X$, and $l\in Y$.
\end{definition}

The modular endomorphism operad of a vector space $V$ with a scalar product has the vector space $V^{\otimes n}$ as its component of genus $g$ and arity $n$. It is endowed with both the partial compositions and the contractions defined using the scalar product on~$V$. The homology groups $H_\bullet(\overline{\mathcal{M}}_{g,n})$ of the Deligne--Mumford--Knudsen moduli space of curves also form a modular operad. In the latter case, the partial compositions are given by the pushforwards of the maps $\rho$ and the contractions are given by the pushforwards of the maps $\sigma$.

The genus $0$ part of a modular operad is faithfully encoded into the following operad.

\begin{proposition}[\cite{GetKap98}] \label{prop:Op-WheeledOp}
Let $(\mathcal{P}_g(n), \circ_i^j, \xi_{ij})$ be a modular operad. The $\mathbb{S}$-module 
\begin{equation}
\mathcal{P}(n):= \mathcal{P}_0(n+1)\cong\mathcal{P}_0([n]),
\end{equation}
and the partial compositions $\circ_i := \circ_i^0 $
define a functor 
$$\mathsf{modular\  operads} \to \mathsf{operads}   \ ,$$
which sends the endomorphism modular operad $V^{\otimes n}$ to the endomorphism operad $\mathrm{End}_V$.
\end{proposition}

The genus $0$ operad $H_\bullet(\overline{\mathcal{M}}_{0,n})$ encodes hypercommutative algebras. It
contains the operad $Com\cong H_0(\overline{\mathcal{M}}_{0,n})$ encoding commutative algebras and it is Koszul dual to the genus $0$ operad $H_\bullet(\mathcal{{M}}_{0,n})$, see \cite{Getzler95}.

\begin{definition}[Wheeled operad \cite{MMS09}]
A datum of a \emph{wheeled operad} is 
 a pair of $\mathbb{S}$-modules, $\lbrace \mathcal{P}(1,n)\rbrace_{n\in \mathbb N}$ and $\lbrace \mathcal{P}(0,n)\rbrace_{n\in \mathbb N}$, endowed with equivariant \emph{partial compositions}
$$\circ_i \ : \ \mathcal{P}(\varepsilon, m)\otimes \mathcal{P}(1, n) \to  \mathcal{P}(\varepsilon, m+n-1), \quad \text{for} \ 1\leq i \leq m \ \text{and} \ 0\leq \varepsilon \leq 1    $$
\begin{center}
\includegraphics[scale=0.2]{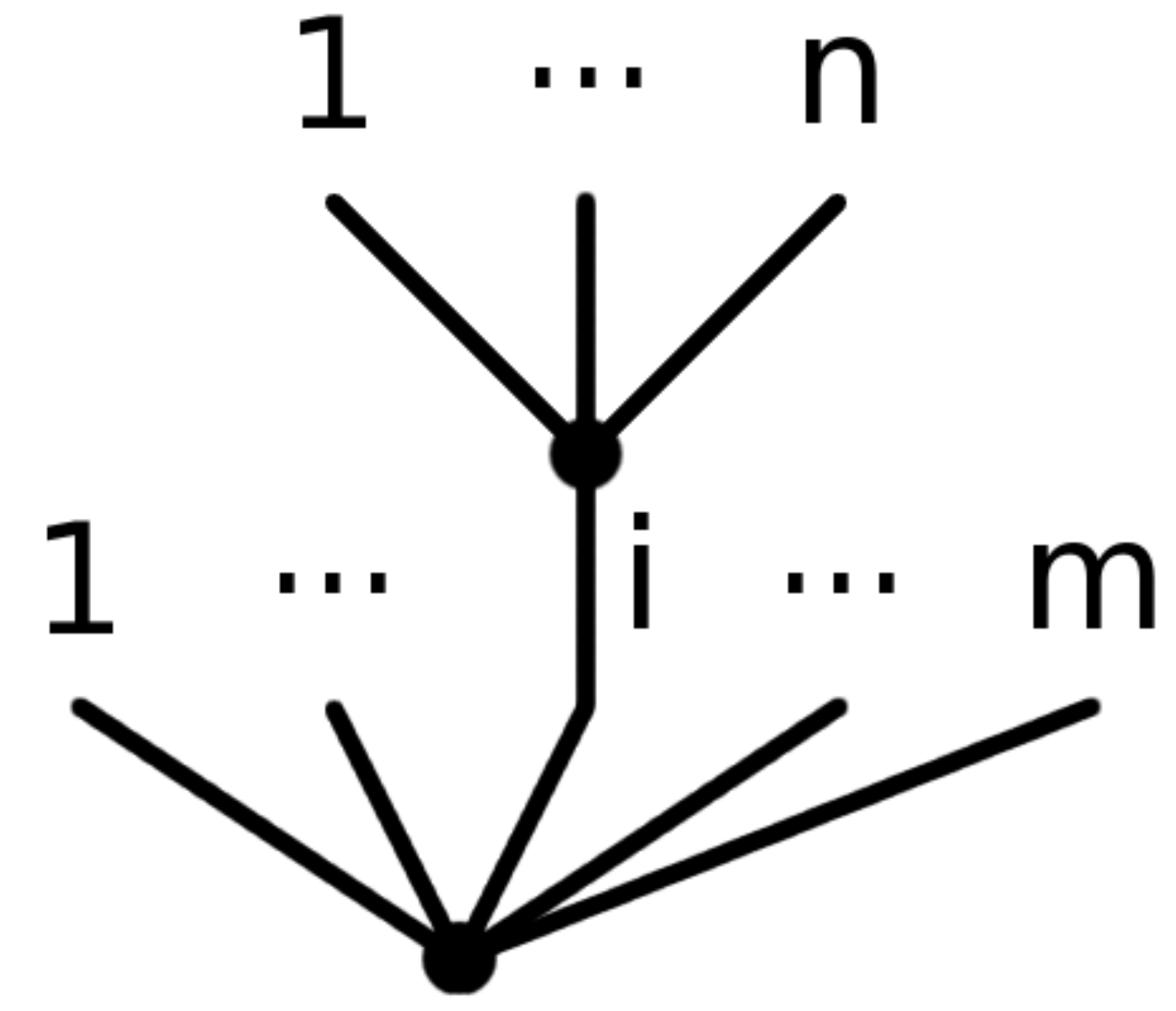} 
\end{center}
and equivariant \emph{wheel contractions}
$$\xi^i \ : \  \mathcal{P}(1, n) \to  \mathcal{P}(0, n-1), \quad \text{for} \ 1\leq i \leq n    \ . $$
\begin{center}
\includegraphics[scale=0.2]{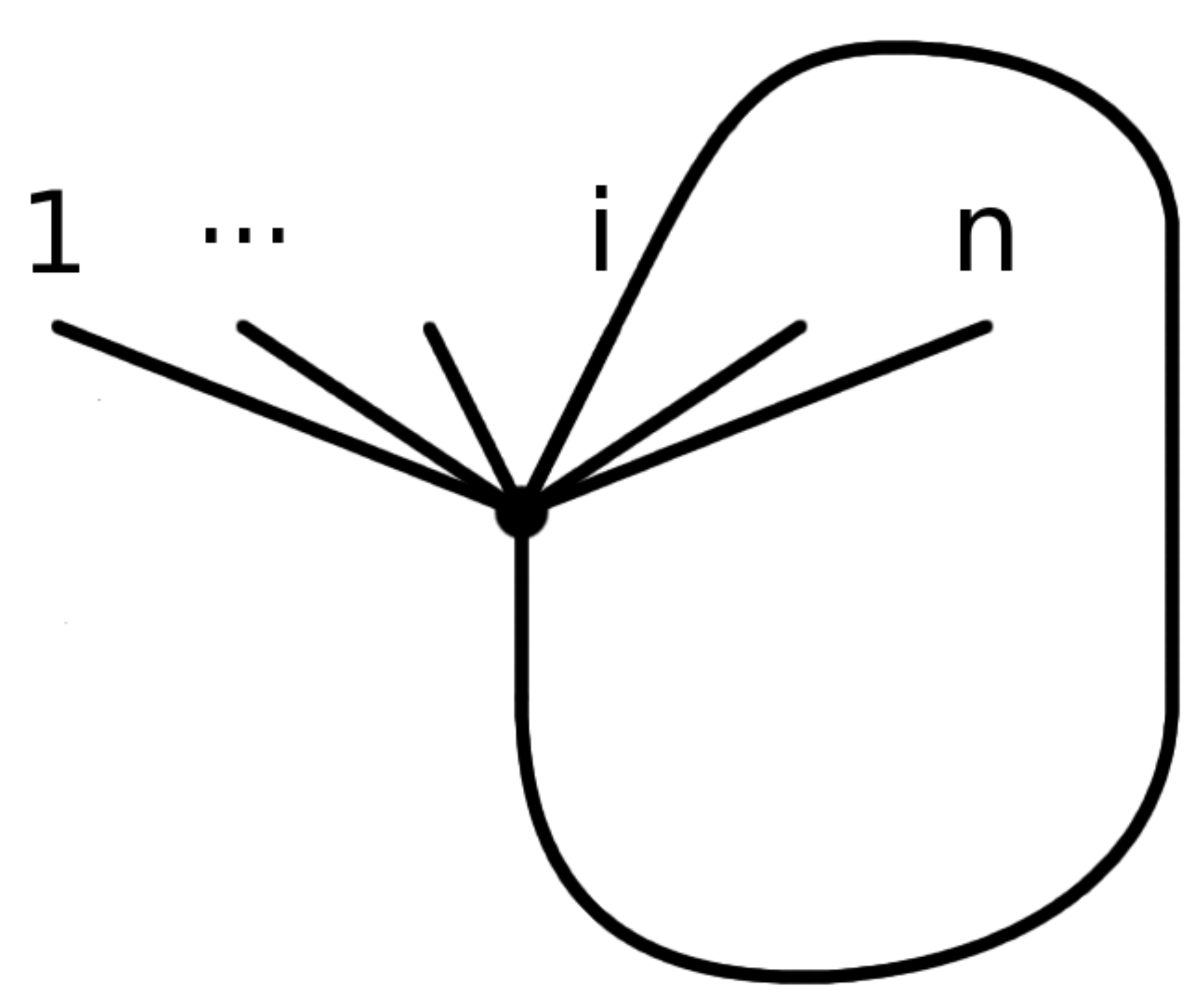} 
\end{center}
These structure maps are required to satisfy the following properties for every choice of $\mu \in \mathcal{P}(\varepsilon, X)$, $\nu\in \mathcal{P}(1, Y)$, and $\omega\in \mathcal{P}(1, Z)$:
\begin{displaymath}
(\mu \circ_i \nu)\circ_j \omega = \left\{ 
\begin{array}{ll}
 \mu \circ_i (\nu\circ_j\omega ), &\mathrm{when}\ j\in Y,  \\
 (\mu \circ_j\omega)\circ_{i}\nu , &\mathrm{when}\  j\in X, \\
\end{array}\right.
\end{displaymath}
for any $i\in X$;
\begin{displaymath}
\xi^{i}(\nu \circ_j \omega) = \left\{ 
\begin{array}{ll}
\xi^{i}(\nu) \circ_j \omega, &\mathrm{when}\ i\in Y,  \\
\xi^j(\omega \circ_i \nu), &\mathrm{when}\ i\in Z,  
\end{array}\right.
\end{displaymath}
for any $j\in Y$.
\end{definition}

The paradigm of wheeled operads is the endomorphism wheeled operad made up of the two components 
$\lbrace \mathrm{Hom}(V^{\otimes n}, V)\rbrace_{n \in \mathbb N}$ and 
$\lbrace \mathrm{Hom}(V^{\otimes n}, \mathbb Q)\rbrace_{n \in \mathbb N}$
associated to a vector space $V$ equipped with a \emph{trace}, that is a map $\mathrm{tr} \, : \, \mathrm{Hom}(V,V) \to \mathbb{C}$ satisfying $\mathrm{tr}([f,g])=0$.

The genera $0$  and $1$ parts of a modular operad are faithfully encoded into the following wheeled operad.

\begin{proposition}\label{prop:ModOp-WheeledOp}
Let $(\mathcal{P}_g(n), \circ_i^j, \xi_{ij})$ be  a modular operad. The pair of  $\mathbb{S}$-modules 
$$\mathcal{P}(1,n):= \mathcal{P}_0(n+1)\cong\mathcal{P}_0([n]), \quad \text{and} \quad 
\mathcal{P}(0,n):= \mathcal{P}_1(n)\cong\ \mathcal{P}_1(\underline{n}),$$
the partial compositions and the wheel contractions 
$\circ_i := \circ_i^0$ and $\xi^i:=\xi_{i0}$
define a functor 
$$\mathsf{modular\  operads} \to \mathsf{wheeled\  operads}   \ , $$
which sends the endomorphism modular operad $V^{\otimes n}$ to the endomorphism  wheeled  operad $\mathrm{End}_V$.
\end{proposition}

\begin{proof}
It is straightforward to check the various axioms of a wheeled operad. Any scalar product $\eta$ of a vector space $V$ induces a trace. Using the scalar product $\eta$, one makes the identifications
 $$
V^{\otimes(n+1)}\cong \mathrm{Hom}(V^{\otimes n}, V) \text{ and } V^{\otimes n} \cong \mathrm{Hom}(V^{\otimes n}, \mathbb{C}), 
 $$
which proves the last statement. 
\end{proof}

Therefore, an algebra over a modular operad, that is a vector space $V$ together with the morphism of modular operads $\mathcal{P}\to \mathrm{End}_V$, induces, under the aforementioned functor, an algebra over the associated wheeled operad. The latter notion is more general than the former one: as opposed to modular operads, a wheeled operad can act on infinite dimensional vector spaces.

\smallskip

We conclude this section with a natural open question.

\begin{question}
Are the two wheeled operads, made up of the genera $0$ and $1$ parts of the modular operads $H_\bullet(\overline{\mathcal{M}}_{g,n})$ and of  $H_\bullet({\mathcal{M}}_{g,n})$, Koszul dual to one another? Are they Koszul wheeled operads? 
\end{question}

\subsection{Cohomological Field Theories}

Roughly speaking, a cohomological field theory on a graded vector space~$V$ is a system of cohomology classes on the moduli spaces of curves with values in tensor powers of $V$ compatible with all natural mappings between the moduli spaces. Expressed formally, this means that a cohomological field theory is a representation of the modular operad formed by the cohomology algebras of moduli spaces of curves. Let us use the previous section to make this precise. To conform with the usual conventions, we use cohomology classes: even though a na\"ive translation of what was just said suggests to think of a cohomological field theory as of a collection of elements $a_{g,n}\colon\Hom(H_\bullet(\overline{\calM}_{g,n}), V^{\otimes n})$, we replace the dual space to the homology by the cohomology, and let our elements belong to $H^\bullet(\overline{\calM}_{g,n})\otimes V^{\otimes n}$.

\begin{definition}[CohFT]
Given a graded vector space $V$ with a basis $\{e_1,\dots, e_s\}$ ($e_1$ plays the role of a unit) with a scalar product $\eta$, a \emph{cohomological field theory} (CohFT) on~$V$ is defined as a system of classes $\alpha_{g,n}\in H^\bullet(\overline{\calM}_{g,n})\otimes V^{\otimes n}$ satisfying the following properties: 
\begin{itemize}
\item[$\diamond$] the classes $\alpha_{g,n}$ is equivariant with respect to the actions of the symmetric group $\mathbb{S}_n$ on the labels of marked points and on the factors of $V^{\otimes n}$.
\item[$\diamond$] the pullbacks via the natural mappings $\sigma$ and $\rho$ correspond to the pairings with $\eta$ of the factors in tensor powers corresponding to the points in the preimage of the node:
\begin{gather}
\sigma^*\alpha_{g,n}=\left(\alpha_{g-1,n+2},\eta^{-1}\right), \\ 
\rho^*\alpha_{g,n}=\left(\alpha_{g_1,n_1+1}\otimes\alpha_{g_2,n_2+1},\eta^{-1}\right). 
\end{gather}
\item[$\diamond$] the unital conditions for the element $e_1$:
\begin{gather}
\left(\alpha_{0,3},e_1\otimes e_i\otimes e_j\right)=\eta(e_i,e_j),\\
\pi^*\alpha_{g,n}=\eta^{(n+1)}\left(\alpha_{g,n+1}, e_1\right)
\end{gather}
(the superscript in $\eta^{(n+1)}$ refers to the fact that we use $\eta$ to compute the scalar product of $e_1$ with the factor of $V^{\otimes n}$ corresponding to the last marked point).
\end{itemize}  
\end{definition}

A \emph{topological field theory} (TFT) is a special case of a cohomological field theory for which all the classes $\alpha_{g,n}$ are of homological degree~$0$, that is belong to $H^0(\overline{\calM}_{g,n})\otimes V^{\otimes n}$. One can easily show that in this case the whole system $\{\alpha_{g,n}\}$ is determined by the class $\alpha_{0,3}\in V^{\otimes 3}$. A CohFT is a TFT if and only if $\alpha_{0,3}$ determines a structure of a unital commutative Frobenius algebra on the vector space $V$, see~\cite{Dub96}.

The definition of a CohFT and thus of a TFT includes an element $e_1$ playing the role of a unit of~$V$. However, this unit is not of an operadic nature (because of the stability condition, $\mathcal{M}_{0,2}=\varnothing$).

Throughout this paper, we shall consider possibly infinite dimensional genus $0$ (respectively genera $0$ and $1$) CohFTs without a unit as algebras over the operad (respectively wheeled operad) associated to the modular operad $H^\bullet(\overline{\mathcal{M}}_{g,n})$. For practical purposes, this means that a genus~$0$ CohFT is encoded by a collection of classes $$\alpha_n:=\alpha_{0,n+1}\in H^\bullet(\overline{\calM}_{0,n+1})\otimes\cEnd_V(n),$$ and a CohFT in the genera $0$ and~$1$ is encoded by a collection of classes 
 $$
\alpha_n=\alpha_{0,n+1}\in H^\bullet(\overline{\calM}_{0,n+1})\otimes\cEnd_V(n)\quad\text{and}\quad\beta_n=\alpha_{1,n}\in H^\bullet(\overline{\calM}_{1,n})\otimes \Hom(V^{\otimes n},\mathbb{C}).
 $$ 
Observe our choice of notation: the class $\alpha_n=\alpha_{0,n+1}$ represents an $n$-ary operation; it is a cohomology class of the moduli space of curves with $n+1$ marked points, since we need $n$ inputs and one output. In the same spirit, $\beta_n=\alpha_{1,n}$ represents an operation with $n$ inputs and 
no outputs.

\subsection{Correlators}\label{sec:correlators}

Using the $\psi$-classes, it is possible to extract certain numerical information from a cohomological field theory. 
\begin{definition}[Correlators]
Let $d_1,\ldots,d_n$ be a sequence of nonnegative integers. The \emph{correlator} $\langle\tau_{d_1}\cdots\tau_{d_n}\rangle_{g,V}$ associated to this sequence is defined by the formula
\begin{equation}\label{correlator-operations}
\langle\tau_{d_1}\cdots\tau_{d_n}\rangle_{g,V}:=\int_{\overline{\calM}_{g,n}}\alpha_{g,n}\cdot\prod_{j=0}^n\psi_j^{d_j}\in V^{\otimes n}. 
\end{equation}
\end{definition}

In the computations below, we shall express the correlators as particular elements of $\End_V$ with correlators for $V=\mathbb{C}$ as coefficients. The latter correlators are just numbers denoted by $\langle\tau_{d_1}\cdots\tau_{d_n}\rangle_g$. For a genus zero topological field theory, it is known that 
\begin{equation}\label{correlator-numbers}
\langle\tau_{d_1}\cdots\tau_{d_n}\rangle_0=
\left\{
\begin{aligned}
\frac{(n-3)!}{d_1!\cdots d_n!}&  \quad \text{ if }d_1+\cdots+d_n=n-3,  \\
 0 \phantom{n-3}&\quad \text{ otherwise.}
\end{aligned}
\right.
\end{equation}
By a standard argument, see for instance~\cite{Sha09}, the correlators contain all the information necessary to compute the restriction of a cohomological field theory on the tautological rings of the moduli spaces.

One of the key technical tools in our proofs will be the following relations, see for instance~\cite{Manin}.

\begin{proposition}[Topological recursion relations] The following relations between the correlators hold:
\begin{itemize}
 \item[$\diamond$] the genus zero relations at the points $i,j,k$
\begin{equation}\label{TRR-0}
\langle\tau_{d_1}\tau_{d_2}\cdots\tau_{d_i+1}\cdots\tau_{d_n}\rangle_0=
\sum_{\substack{I\sqcup J=\underline{n}\\ i\in I; \, j,k\in J}}
\langle\tau_{d_I}\tau_0\rangle_0\langle\tau_0\tau_{d_J}\rangle_0.
\end{equation}
 \item[$\diamond$] a more symmetric version of the genus zero relations at the points $i,j$ that follows immediately from the previous ones 
\begin{equation}\label{TRR-0-sym}
\langle\tau_{d_1}\tau_{d_2}\cdots\tau_{d_i+1}\cdots\tau_{d_n}\rangle_0+
\langle\tau_{d_1}\tau_{d_2}\cdots\tau_{d_j+1}\cdots\tau_{d_n}\rangle_0=
\sum_{\substack{I\sqcup J=\underline{n}\\ i\in I, \, j\in J}}
\langle\tau_{d_I}\tau_0\rangle_0\langle\tau_0\tau_{d_J}\rangle_0,
\end{equation}
 \item[$\diamond$] the genus~$1$ relations at the point~$i$
\begin{equation}\label{TRR-1}
\langle\tau_{d_1}\tau_{d_2}\cdots\tau_{d_i+1}\cdots\tau_{d_n}\rangle_1=
\sum_{\substack{I\sqcup J=\underline{n}, \\ i\in I}}
\langle\tau_{d_I}\tau_0\rangle_0\langle\tau_0\tau_{d_J}\rangle_1+
\frac{1}{24}\langle\tau_{d_1}\tau_{d_2}\cdots\tau_{d_i}\cdots\tau_{d_n}\tau_0^2\rangle_0.
\end{equation}
\end{itemize}
\end{proposition}

\subsection{Givental group action}

In this section, we describe a part of the Givental theory of a certain group action on CohFTs~\cite{Giv01,Giv04}.

\begin{definition}[Givental group and Lie algebra]
Let $V$ be a vector space with a scalar product $\eta$ as above. The group that plays the key role in Givental's construction is the (upper triangular group of the) group of symplectomorphisms of Laurent series with coefficients in~$V$. It is the group of formal power series with coefficients in the space of endomorphisms of~$V$ consisting of all series $R(z)=Id+R_1z+R_2z^2+\dots$ satisfying $R^*(-z)R(z)=Id$. Its Lie algebra consists of all series $r_1z+r_2z^2+\dots$, where $r_l\in\End(V)$ is symmetric for odd $l$ and skew-symmetric for even $l$ (with respect to the scalar product $\eta$).
\end{definition}

Following~\cite{Kaz07,Teleman}, we associate to an element $\sum_{l=1}^{\infty}r_lz^l$ as above an infinitesimal deformation of CohFT.

\begin{definition}[Givental Lie algebra action]
The Givental Lie algebra action on cohomological field theories takes the system of classes $\alpha_{g,n}\in H^\bullet(\overline{\calM}_{g,n})\otimes V^{\otimes n}$ to the system of classes $(\widehat{r_lz^l}.\alpha)_{g,n}\in H^\bullet(\overline{\calM}_{g,n})\otimes V^{\otimes n}$ given by the formula 
\begin{multline}\label{Givental-general}
(\widehat{r_lz^l}.\alpha)_{g,n}:=
-\left(\pi_*(\alpha_{g,n+1}\cdot\psi_{n+1}^{l+1}),r_l(e_1)\right)+\sum_{m=1}^n(\alpha_{g,n}\cdot\psi_m^l)\circ_m r_l+\\+
\frac12\sum_{i=0}^{l-1}(-1)^{i+1}\left(\sigma_*(\alpha_{g-1,n+2}\cdot\psi_{n+1}^i\psi_{n+2}^{l-1-i}),\eta^{-1}r_l\right)+\\+
\frac12\sum_{\substack{I\sqcup J=\underline{n},\\ i+j=l-1,\\ g_1+g_2=g}}(-1)^{j+1}
\left(\rho_*(\alpha_{g_1,|I|+1}\cdot\psi_{|I|+1}^i\otimes\alpha_{g_2,|J|+1}\cdot\psi_{|J|+1}^j),\eta^{-1}r_l\right).
\end{multline}
In this formula, we assume that the points in the preimage of the node are the points labelled $n+1$ and $n+2$ in the second sum, and the points $|I|+1$ on $\overline{\calM}_{g_1+1,|I|+1}$ and $|J|+1$ on $\overline{\calM}_{g_2+1,|J|+1}$ in the third sum. 
\end{definition}

\begin{proposition}[\cite{Kaz07,Teleman}]
The classes  
\begin{equation}
\tilde{\alpha}_{g,n}:=\left(\exp\left(\sum_{l=1}^\infty \widehat{r_lz^l}\right).\alpha\right)_{g,n} 
\end{equation}
are well-defined cohomology classes with the values in the tensor powers of~$V$ that define a CohFT. 
\end{proposition}

If one chooses to ignore the unit $e_1$, it is possible: the exact same formulae with the summand containing $e_1$ omitted provide a well defined Lie algebra action on CohFTs, see for example~\cite{KMS11}.
We shall be using Formula~\eqref{Givental-general} only in the cases of genera~$0$ and~$1$ and without a unit, as follows. To match the notation of the subsequent sections, we shall write the formula for the action of the operator $D_lz^{l-1}$,  $l\ge 1$. Note that Formula~\eqref{Givental-general} makes sense even for $l=0$, in which case we get the usual action on the Lie algebra $\End(V)$ by the Leibniz rule.

In the case of genus~$0$, we have the action
\begin{equation}
\widehat{D_lz^{l-1}}.\{\alpha\}=\{\alpha'\},
\end{equation}
and the general formula~\eqref{Givental-general} simplifies to
\begin{multline}\label{Givental-action-genus-0}
\alpha'_n=
(-1)^lD_l\circ_1\alpha_{n}\cdot\psi_0^{l-1}+\sum_{m=1}^n\alpha_{n}\cdot\psi_m^{l-1}\circ_m D_l+\\+
\sum_{\substack{I\sqcup J=\underline{n}, |I|\ge2,\\ i+j=l-2}}(-1)^{i+1}(\alpha_{|J|+1}\cdot\psi_1^j) \tilde{\circ}_1  (D_l \circ_1\alpha_{|I|}\cdot \psi_0^i).
\end{multline}
Here we assume that the output of every operadic element corresponds to the marked point $x_0$ on the curve, and that, in the last sum, the points in the preimage of the node are the point $x_0$ on the curve with $|I|+1$ marked points and the point $x_1$ on the curve with $|J|+2$ marked points. The operation $\tilde{\circ}_1$ in the last sum refers to using the push-forward $\rho_*$ on the cohomology and, simultaneously, the composition $\circ_1$ in the endomorphism operad. Because of our assumption on the (skew-)symmetry of components of power series, the first summand acquires the sign $(-1)^l$: our translation into the operadic language requires us to identify $V$ with its dual using the scalar product $\eta$, so we have to replace the operator $D_l$ acting on the output by its adjoint. After that replacement, the assumptions on the (skew-)symmetry of components of power series do not need to be used anymore: any element of $z\End(V)[[z]]$ provides a well defined infinitesimal deformation of a given CohFT,
 viewed as an algebra over the corresponding operad~\cite{KMS11,SZ}.

In the case of genus~$1$, we have the action
\begin{equation}
\widehat{D_lz^{l-1}}.\{\alpha,\beta\}=\{\alpha',\beta'\}.
\end{equation}
The classes $\alpha'_n$ are defined by the genus zero formula~\eqref{Givental-action-genus-0}. For the $\beta$-classes, we can simplify the general formula~\eqref{Givental-general} to
\begin{multline}\label{Givental-action-genus-1}
\beta'_n=
\sum_{m=1}^n\beta_n\cdot\psi_m^{l-1}\circ_mD_l+
\sum_{\substack{I\sqcup J=\underline{n}, |I|\ge2,\\ i+j=l-2}}(-1)^{i+1}(\beta_{|J|+1}\cdot\psi_1^j) \tilde{\circ}_1  (D_l\circ_1\alpha_{|I|}\cdot\psi_0^i)+\\+
\frac12\sum_{i+j=l-2}(-1)^{i+1}\tilde{\xi}_{n+1}(\alpha_{n+1}\cdot\psi_0^i\cdot\psi_{n+1}^j\circ_{n+1} D_l).
\end{multline}
Here we assume that the points in the preimage of the node are the point $x_0$ on the curve of genus~$0$ and the point $x_1$ on the curve of genus~$1$ in the second sum, and the points $x_0,x_{n+1}$ in the third sum. The operation $\tilde{\circ}_1$ in the second sum refers to using the push-forward $\rho_*$ on the cohomology and, simultaneously, the composition $\circ_1$ in the wheeled endomorphism operad. The operation $\tilde{\xi}_{n+1}$ in the last sum refers to using the push-forward $\sigma_*$ on the cohomology and, simultaneously, the contraction $\xi_{n+1}$ in the endomorphism operad.  As above, once we pass to wheeled operads, we may drop the assumptions on components of power series: any element of $z\End(V)[[z]]$ provides a well defined infinitesimal deformation of a given CohFT, viewed as an algebra over the corresponding wheeled operad.

\section{Differential operators}\label{sec:DifOp}

In this section, we recall the definition and the properties of differential operators on commutative algebras, and introduce a new notion of compatibility with the trace. Several new results are proved, both for classical differential operators and for operators compatible with the trace.

\begin{definition}[Order of an operator~\cite{Koszul85}]
Let $D$ be a linear map on a commutative algebra $V$. Let us define the \emph{Koszul bracket hierarchy} 
\begin{equation}
\langle-,-,\ldots,-\rangle_l^D\colon V^{\otimes l}\to V
\end{equation}
by $\langle f\rangle_1^D:=D(f)$ and, recursively, by
\begin{multline}\label{koszul-recursive}
\langle f_1,\ldots,f_{l-1},f_l,f_{l+1}\rangle_{l+1}^D=\\ 
 \langle f_1,\ldots,f_{l-1}, f_lf_{l+1}\rangle_l^D
-\langle f_1,\ldots,f_{l-1},f_l\rangle_l^D f_{l+1}
-f_l\langle f_1,\ldots,f_{l-1},f_{l+1}\rangle_l^D.
\end{multline}
The operator $D$ is said to be a \emph{differential operator on~$V$ of order at most~$l$} if the bracket $\langle-,-,\ldots,-\rangle_{l+1}^D$ is identically equal to zero. 
\end{definition}

We denote by $D\in\Diff_{\le l}(V)$ the set of all differential operators of order at most~$l$, and by $\Diff(V)$ the set of differential operators of all possible orders: 
\begin{equation}
\Diff(V):=\bigcup_{l\ge0}\Diff_{\le l}(V).
\end{equation}
The composition of differential operators makes $\Diff(V)$ an associative algebra. This algebra is filtered by the order of operators: 
\begin{equation}\label{filtered-algebra}
\Diff_{\le k}(V)\circ\Diff_{\le l}(V)\subset\Diff_{\le k+l}(V). 
\end{equation}

There are many definitions of differential operators on commutative associative algebras that can be found in the literature. The oldest and the most known one is probably due to Grothendieck \cite{Grothendieck67}, requiring, for an operator of order at most~$l$, that
\begin{equation}
[[\ldots[[D,f_1\cdot(-)],f_2\cdot(-)],\ldots],f_{l+1}\cdot(-)]=0, 
\end{equation}
where each $f_i\cdot(-)$ stands for the operator $g\mapsto f_i\cdot g$. Later, other definitions have appeared, see for example~\cite{Koszul85,Akman97}. For algebras with a unit, and operators which annihilate the unit, all these definitions are equivalent to each other, see~\cite{AkmanIonescu08}.

It is easy to see that the above recursive definition of the Koszul hierarchy results in the following explicit formula for the higher brackets:
\begin{equation}
 \langle f_1,\ldots,f_{l-1}, f_l\rangle_l^D=
\sum_{\substack{
I\subset\underline{l},\\
|I|\ge1}}
(-1)^{l-|I|} D(f_I)
\cdot f_{\underline{l}\setminus I}.
\end{equation}
From this formula, it is clear that the Koszul brackets are graded symmetric, and the condition of being a differential operator of order at most~$l$ becomes
\begin{equation}\label{order-l-definition-1}
D(f_1f_2\cdots f_{l+1})=\sum_{\substack{I\subset\underline{l+1},\\ 1\le|I|\le l}}(-1)^{l-|I|}D(f_I)f_{\underline{l+1}\setminus I}.
\end{equation}

From calculus, it is well known that to compute the derivative (respectively the second derivative) of a product of~$n$ factors, it is enough to know the derivative of each factor (respectively the second derivative of each factor and each product of two factors):
\begin{gather}
(f_1f_2\cdots f_n)'=\sum_{1\le i\le n} f_i'f_{\underline{n}\setminus \{i\}}, \label{order-1-identity}\\
(f_1f_2\cdots f_n)''=\sum_{1\le i<j\le n} (f_if_j)''f_{\underline{n}\setminus \{i,j\}}-(n-2)\sum_{1\le i\le n} f_i''f_{\underline{n}\setminus \{i\}}. \label{order-2-identity}
\end{gather}
We shall need the following generalisation of these formulae for any commutative algebra and any order of an operator.
\begin{proposition}\label{identities-diff}
For a differential operator~$D$ of order at most~$l$, and for every $n\ge l+1$ we have
\begin{equation}\label{order-l-identity}
D(f_1f_2\cdots f_n)=\sum_{\substack{I\subset\underline{n},\\ 1\le |I|\le l}}(-1)^{l-|I|}\binom{n-1-|I|}{l-|I|}D(f_I)f_{\underline{n}\setminus I}. 
\end{equation} 
\end{proposition}

\begin{proof}
Let us prove this formula by induction on $n-l$. For the basis of the induction, that is $n-l=1$, Formula~\eqref{order-l-identity} coincides literally with Formula~\eqref{order-l-definition-1}, so there is nothing to prove. For inductive step, let us note that a differential operator
of order at most~$l$ is automatically a differential operator of order at most~$l+1$, therefore by the induction hypothesis,
\begin{equation}\label{order-l+1-identity}
D(f_1f_2\cdots f_n)=\sum_{\substack{J\subset\underline{n},\\ 1\le|J|\le l+1}}(-1)^{l+1-|J|}\binom{n-1-|J|}{l+1-|J|}D(f_J)f_{\underline{n}\setminus J}. 
\end{equation}
To show that this identity is equivalent to \eqref{order-l-identity}, let us eliminate all the terms $D(f_J)f_{\underline{n}\setminus J}$ with $|J|=l+1$ using \eqref{order-l-definition-1}. This would replace these terms by a sum of terms of the form $(-1)^{k-|I|}D(f_I)f_{\underline{n}\setminus I}$, and each of the latter terms will appear exactly $\binom{n-|I|}{l+1-|I|}$ times (the number of choices of the ``new'' factors in $f_{\underline{n}\setminus I}$). Therefore, we have 
\begin{multline}
D(f_1f_2\cdots f_n)=
\sum_{I\subset\underline{n}, |I|\le l}(-1)^{l+1-|I|}\binom{n-1-|I|}{l+1-|I|}D(f_I)f_{\underline{n}\setminus I}+\\+
\sum_{I\subset\underline{n}, |I|\le l}(-1)^{l-|I|}\binom{n-|I|}{l+1-|I|}D(f_I)f_{\underline{n}\setminus I}=\\=
\sum_{I\subset\underline{n}, |I|\le l}(-1)^{l-|I|}\binom{n-1-|I|}{l-|I|}D(f_I)f_{\underline{n}\setminus I}, 
\end{multline}
which is exactly what we wanted to prove. 
\end{proof}

The key property of higher brackets that we shall need in this paper is given by the following proposition. 
\begin{proposition}[\cite{Bering96}]
We have 
\begin{multline}\label{commutator-Koszul}
\langle f_1,\ldots,f_{n-1}, f_n\rangle_n^{[A,B]}=\\= 
\sum_{\substack{I\sqcup J = \underline{n},\\ |I|=r\ge1}} 
\langle \langle f_{i_1},\ldots,f_{i_r}\rangle_r^{B},f_{j_1},\ldots,f_{j_{n-r}}\rangle_{n-r+1}^A-\\-
(-1)^{\deg(A)\deg(B)}\langle \langle f_{i_1},\ldots,f_{i_r}\rangle_r^{A},f_{j_1},\ldots,f_{j_{n-r}}\rangle_{n-r+1}^B.
\end{multline}
\end{proposition}

This formula immediately implies the well known fact that the commutator ``makes the order of differential operators drop by~$1$'':
\begin{equation}\label{filtered-Lie}
[\Diff_{\le k}(V),\Diff_{\le l}(V)]\subset\Diff_{\le k+l-1}(V). 
\end{equation}
Further in this paper, we shall need a modification of this property for commutative algebras with traces which we shall now present.

\begin{definition}
A commutative algebra~$V$ is said to be an \emph{algebra with a trace} if it is equipped with the linear functional $\str\colon\End(V)\to\mathbb{C}$ that vanishes on commutators, that is $\str([A,B])=0$ for all $A,B\in\End(V)$. 
\end{definition}

\begin{proposition}\label{order-wheel}
A differential operator $A$ of order at most $l$ on a commutative algebra~$V$ with a trace satisfies the trace identity
\begin{equation}
\str\left(\langle f_1,\ldots,f_{l-1}, f_l\rangle_l^A\cdot(-)\right)=0, \label{drop-order-1}
\end{equation}
for all $f_1,\ldots,f_l\in V$.
\end{proposition}

\begin{proof}
Since $A$ is an operator of order at most~$l$, the operator $\langle f_1,\ldots, f_l,-\rangle_{l+1}^A$ vanishes identically. This means that
\begin{multline}
0=\langle f_1,\ldots,f_{l-1},f_l,-\rangle_{l+1}^A=\\ 
 =\langle f_1,\ldots,f_{l-1}, f_l\cdot(-)\rangle_l^A
-\langle f_1,\ldots,f_{l-1},f_l\rangle_l^A\cdot(-)
-f_l\langle f_1,\ldots,f_{l-1},-\rangle_l^A.
\end{multline}
As a consequence, we see that
\begin{equation}
\langle f_1,\ldots,f_{l-1},f_l\rangle_l^A\cdot(-)=[f_l\cdot(-),\langle f_1,\ldots,f_{l-1},-\rangle_l^A],
\end{equation}
so the trace in question vanishes being the trace of a commutator.
\end{proof}

Informally, this result means that ``the order of a differential operator drops by~$1$ in the presence of the trace''. A subclass of differential operators of particular interest to us consists of operators whose order drops unexpectedly low in the presence of the trace. 

\begin{definition}
Let $V$ be a commutative algebra with a trace. A differential operator~$A$ of order at most $l\ge 3$ is said to be \emph{strongly compatible with the trace} if it satisfies the trace identity
\begin{equation}
\str\left(\langle f_1,\ldots,f_{l-1}\rangle_{l-1}^A\cdot(-)\right)=0,\label{strong-compat-1}
\end{equation}
for all $f_1,\ldots,f_{l-1}\in V$. Notation: $A\in\Diff_{\le l}^{\str}(V)$, $\Diff^{\str}(V):=\bigcup_{l\ge3}\Diff^{\str}_{\le l}(V)$
\end{definition}

Note that from the defining property $\langle f_1,\ldots,f_{l-1}, f_l,f_{l+1}\rangle_{l+1}^A=0$ of a differential operator one can instantly deduce $\langle f_1,\ldots,f_{l-1}, f_l,f_{l+1}\rangle_{l+1}^Af_{l+2}=0$ for every $f_{l+2}\in V$. In the presence of traces, the corresponding conclusion is not obvious, and we shall prove it now. It will be frequently used in our proofs in the main part of the article.

\begin{proposition}
Let $V$ be a commutative algebra with a trace. Suppose that $A$ is a differential operator of order at most $l\ge 3$ which is strongly compatible with the trace. Then 
\begin{gather}
\str\left(\langle f_1,\ldots,f_{l-1}, f_l\rangle_l^A f_{l+1}\cdot(-)\right)=0, \label{drop-order-2}\\
\str\left(\langle f_1,\ldots,f_{l-1}\rangle_{l-1}^A f_l\cdot(-)\right)=0 \label{strong-compat-2}
\end{gather}
for all $f_1,\ldots,f_l\in V$. The first of these statements holds for $l=1$ and $l=2$ as well. 
\end{proposition}

\begin{proof}
Let us prove the second statement; the first one is proved analogously to how it is done in Proposition~\ref{order-wheel}. Since $A$ is of order at most $l$, we have
 $$
\str\left(\langle f_1,\ldots,f_{l-1}, f_l\rangle_l^A\cdot(-)\right)=0 
 $$
for all $f_1,\ldots,f_{l-1},f_l\in V$.  
Because of the definition of the Koszul hierarchy, we have 
\begin{multline}
\str\left(\langle f_1,\ldots,f_{l-1}, f_l\rangle_l^A\cdot(-)\right)=\\=
\str\left(\langle f_1,\ldots,f_{l-2}, f_{l-1}f_l\rangle_{l-1}^A\cdot(-)\right)-\str\left(\langle f_1,\ldots,f_{l-2},f_{l-1}\rangle_{l-1}^A f_{l}\cdot(-)\right)-\\
-\str\left(f_{l-1}\langle f_1,\ldots,f_{l-2},f_l\rangle_{l-1}^A\cdot(-)\right).
\end{multline}
so since $A$ is strongly compatible with the trace we have 
\begin{equation}
\str\left(\langle f_1,\ldots,f_{l-2},f_{l-1}\rangle_{l-1}^A f_{l}\cdot(-)\right)+\str\left(f_{l-1}\langle f_1,\ldots,f_{l-2},f_l\rangle_{l-1}^A\cdot(-)\right)=0.
\end{equation}
Recalling that Koszul brackets are symmetric in their arguments, we see that
\begin{multline}
\str\left(\langle f_1,\ldots,f_{l-2},f_{l-1}\rangle_{l-1}^A f_{l}\cdot(-)\right)=-\str\left(f_{l-1}\langle f_1,\ldots,f_{l-2},f_l\rangle_{l-1}^A\cdot(-)\right)=\\=
\str\left(f_{l-2}\langle f_1,\ldots,f_{l-1},f_l\rangle_{l-1}^A\cdot(-)\right)=-\str\left(f_{l}\langle f_1,\ldots,f_{l-2},f_{l-1}\rangle_{l-1}^A\cdot(-)\right),
\end{multline}
so 
\begin{equation}
\str\left(\langle f_1,\ldots,f_{l-2},f_{l-1}\rangle_{l-1}^A f_{l}\cdot(-)\right)=0,
\end{equation}
which is what we need.
\end{proof}

\begin{proposition}\label{lie-subalgebra}
The subspace $\Diff^{\str}(V)\subset \Diff(V)$ is a Lie subalgebra. 
\end{proposition}

\begin{proof}
Assume that $A\in\Diff_{\le k}^{\str}(V)$ and $B\in\Diff_{\le l}^{\str}(V)$. Let us examine the trace 
\begin{equation}
\str\left(f\langle f_1,\ldots,f_{k+l-2}\rangle_{k+l-2}^{[A,B]}\cdot(-)\right),
\end{equation}
which we rewrite, using Formula~\eqref{commutator-Koszul}, as
\begin{multline}
\sum_{\substack{I\sqcup J = \underline{k+l-2},\\ |I|=r\ge1}} 
\str\left(f\langle \langle f_{i_1},\ldots,f_{i_r}\rangle_r^{B},f_{j_1},\ldots,f_{j_{k+l-2-r}}\rangle_{k+l-1-r}^A\cdot(-)\right)-\\-
(-1)^{\deg(A)\deg(B)}\str\left(f\langle \langle f_{i_1},\ldots,f_{i_r}\rangle_r^{A},f_{j_1},\ldots,f_{j_{k+l-2-r}}\rangle_{k+l-1-r}^B\cdot(-)\right).
\end{multline}
The first trace vanishes for $r\ge l+1$ because the operator $B$ is of order at most $l$ and also for $k+l-1-r\ge k-1$ (that is, for $l\ge r$) because $A$ is of order at most~$k$ and is strongly compatible with the trace. The second trace vanishes for $r\ge k+1$ because $A$ is of order at most $k$ and for $k+l-1-r\ge l-1$ (that is, for $k\ge r$) because $B$ is of order at most~$l$ and is strongly compatible with the trace. \end{proof}

\section{Stabilisers in genus~0}\label{com-bv-infty-via-givental}

In this section, we make a very particular choice of a CohFT: we consider CohFTs in genus zero, and moreover, we only work with TFTs, that is we assume we are given the classes
\begin{equation}\label{dimension-0}
 \alpha_n\in H^0(\overline{\calM}_{0,n+1},\mathbb{C})\otimes\cEnd_V(n).  
\end{equation}
The class $\alpha_2$ endows $V$ with a structure of a graded commutative associative algebra. To simplify the notation, we denote $\alpha_2(x,y)$ by $x\cdot y$, and by~$xy$.

Let $\mathbf{D}=\sum_{l\ge1}D_lz^{l-1}\in\End(V)[[z]]$ be an element of the Givental Lie algebra. It is natural to ask when such an element preserves a given cohomological field theory.

\begin{definition}[Commutative $\BV_\infty$-algebra~\cite{Kra00}]\label{comm-homotopy-bv}
A commutative algebra $V$ is said to be a \emph{commutative $\BV_\infty$-algebra} if it is equipped with a collection of operators $D_l$, $l\ge1$, each $D_l$ being a differential operator of order at most~$l$ and of degree~$2l-3$, such that  $\left(\sum\limits_{l\ge1}D_l\right)^2=0$.
\end{definition}

\begin{definition}[Chain multicomplex~\cite{Meyer}]
A \emph{chain multicomplex} is a graded vector space $V$ together with a system of operators $D_l\colon V_p\to V_{p+2l-3}$ satisfying the condition 
\begin{equation}
\left(\sum_{l\ge1}D_l\right)^2=0,
\end{equation}
or, after separating the homogeneous components,
\begin{equation}\label{resolving-delta}
\sum_{i+j=n}D_iD_j=0.  
\end{equation}
\end{definition}

\begin{remark}
This is the kind of structure that induces spectral sequences, see for example~\cite[10.3.16]{LodayVallette10}. We are using the term ``chain multicomplex'' somewhat loosely: usually chain multicomplexes are assumed to possess additional bi-gradings which are compatible with the degrees of the operators~$D_i$, while our chain multicomplexes look like bigraded chain multicomplexes with respect to the grading $\deg(V_{p,q})=q-p$.
\end{remark}

We are now ready to formulate the main result of this section. 

\begin{theorem}\label{comm-bv-infty}
Let $V$ be a graded vector space. The Givental action of the element $$\mathbf{D}=\sum_{l\ge1}D_lz^{l-1}\in\End(V)[[z]]$$ preserves a given $V$-valued TFT in genus~$0$ if and only if each linear operator $D_k$ is a differential operator of order at most $k$ on the commutative algebra~$V$.
\end{theorem}

\begin{corollary}\label{comm-bv-infty-cor}
Let $V$ be a chain multicomplex with the structure maps $D_l$, $l\ge 1$. The Givental action of the element $\mathbf{D}=\sum_{l\ge1}D_lz^{l-1}\in\End(V)[[z]]$ preserves a given $V$-valued TFT in genus~$0$ if and only if the commutative algebra~$V$ equipped with the operators $D_i$ is a commutative~$\BV_\infty$-algebra.
\end{corollary}

\begin{proof}
First of all, let us note that because of the choice \eqref{dimension-0}, the images $\widehat{D_lz^{l-1}}.\alpha_n$ of the classes $\alpha_n$ under the action of the individual components of~$\mathbf{D}$ are in different cohomological degrees, so the Givental action of $\mathbf{D}$ preserves a TFT if and only if its components preserve it, that is
\begin{equation}\label{preserving-CohFT-individually}
\widehat{D_lz^{l-1}}.\alpha_n=0, \quad \text{ for } l\ge1.
\end{equation}
Note that for $n<l+1$ these conditions are satisfied trivially for cohomological degree reasons, since $\dim_\mathbb{C}\overline{\calM}_{0,n+1}=n-2$.

Second, we shall express our conditions via correlators: taking a sequence of nonnegative integers $d_0,d_1,\ldots,d_n$ with 
$l-1+d_0+d_1+\cdots+d_n=n-2$, 
we can compute the intersection of $\widehat{D_lz^{l-1}}.\alpha_n$ with $\prod_{m=0}^n\psi_m^{d_m}$, thus obtaining the condition on an element of $\cEnd_V(n)$ instead. To take care of signs in subsequent computations, we introduce the $n$-ary operations 
\begin{equation}
F_l(\begin{smallmatrix}d_1&d_2&\cdots&d_n\\ -&-&\cdots&-\end{smallmatrix}; d_0)=(-1)^l\int_{\overline{\mathcal{M}}_{0,n+1}} \widehat{D_lz^{l-1}}.\alpha_n\cdot \prod_{m=0}^n\psi_m^{d_m}.
\end{equation}
Substituting elements $f_1,\ldots,f_n\in V$ in such an operation, we get the system of identities
\begin{equation}\label{numerical-condition-d}
F_l(\begin{smallmatrix}d_1&d_2&\cdots&d_n\\ f_1&f_2&\cdots&f_n\end{smallmatrix}; d_0)=0, \quad f_1,\ldots,f_n\in V. 
\end{equation}
In fact, Formula~\eqref{Givental-action-genus-0} implies that
\begin{multline}\label{identity-d}
F_l(\begin{smallmatrix}d_1&d_2&\cdots&d_n\\ f_1&f_2&\cdots&f_n\end{smallmatrix}; d_0)=
\langle\tau_{d_0+l-1}\tau_{d_1}\cdots\tau_{d_n}\rangle_0D_l(f_1f_2\cdots f_n)+\\+ 
(-1)^l\sum_{m=1}^n \langle\tau_{d_0}\tau_{d_1}\cdots\tau_{d_m+l-1}\cdots\tau_{d_n}\rangle_0f_1\cdots D_l(f_m)\cdots f_n+\\+
\sum_{\substack{I\sqcup J=\underline{n}, |I|\ge2,\\ i+j=l-2}}(-1)^{j+1}\langle\tau_{i}\tau_{d_I}\rangle_0
\langle\tau_{d_0}\tau_{j}\tau_{d_{J}}\rangle_0 D_l(f_I) f_J.
\end{multline}

Showing, for every fixed $l$, that the system of identities \eqref{numerical-condition-d} means precisely that $D_l$ is a differential operator of order at most~$l$ goes in several steps. First of all, we check that the identity $F_l(\begin{smallmatrix}0&0&\cdots&0\\ f_1&f_2&\cdots&f_n\end{smallmatrix}; 0)=0$ is precisely the definition of a differential operator of order at most~$l$ for $n=l+1$. Second, we check that for each~$n>l+1$ the ``most symmetric'' identity $F_l(\begin{smallmatrix}0&0&\cdots&0\\ f_1&f_2&\cdots&f_n\end{smallmatrix}; n-l-1)=0$ coincides with Identity~\eqref{order-l-identity}. To complete the proof, we use the genus~$0$ topological recursion relations to prove the remaining identities by induction.

Let us use Formula \eqref{correlator-numbers} to make the formula \eqref{identity-d} more explicit for $d_1=\cdots=d_n=0$ (which is the case for the first two steps outlined above). Since $\langle\tau_{p}\tau_{n-2-p}\tau_0^{n-2}\rangle_0=\binom{n-2}{p}$, we see that
\begin{multline}
F_l(\begin{smallmatrix}0&0&\cdots&0\\ f_1&f_2&\cdots&f_n\end{smallmatrix}; d_0)=
D_l(f_1f_2\cdots f_n)+\\+ 
(-1)^l\binom{n-2}{l-1}\sum_{m=1}^n f_1\cdots D_l(f_m)\cdots f_n+
\sum_{\substack{I\sqcup J=\underline{n}, |I|\ge2, \\ i=|I|-2,\\ j+d_0=|J|-1}}(-1)^{j+1}\binom{|J|-1}{j} D_l(f_I) f_J,
\end{multline}
which, after putting $d_0=n-l-1$ and eliminating $J$ from the notation, becomes
\begin{equation}
D_l(f_1f_2\cdots f_n)=\sum_{I\subset\underline{n}, |I|\le l}(-1)^{l-|I|}\binom{n-1-|I|}{l-|I|}D_l(f_I)f_{\underline{n}\setminus I},
\end{equation}
which is Identity~\eqref{order-l-identity}. Therefore the identity 
\begin{equation}
F_l(\begin{smallmatrix}0&0&\cdots&0\\ f_1&f_2&\cdots&f_n\end{smallmatrix}; 0)=0 
\end{equation}
coincides with Formula \eqref{order-l-definition-1} that
defines a differential operator of order at most~$l$, and the identities 
\begin{equation}
F_l(\begin{smallmatrix}0&0&\cdots&0\\ f_1&f_2&\cdots&f_n\end{smallmatrix}; n-l-1)=0 
\end{equation}
follow from it, see Proposition~\ref{identities-diff}.

It remains to complete the last step of the proof: the remaining identities follow from the identities $F_l(\begin{smallmatrix}0&0&\cdots&0\\ f_1&f_2&\cdots&f_n\end{smallmatrix}; n-l-1)=0$. This will be deduced from the following lemma.

\begin{lemma}
We have
\begin{multline}\label{recursion-on-identities}
F_l(
\begin{smallmatrix}
d_1+1&d_2&\cdots&d_n\\ 
f_1&f_2&\cdots&f_n
\end{smallmatrix}; 
d_0
)+ 
F_l(
\begin{smallmatrix}
d_1&d_2&\cdots&d_n\\ 
f_1&f_2&\cdots&f_n
\end{smallmatrix}; 
d_0+1
)=\\=
\sum_{\substack{1\in I\subset\underline{n},\\ d_{i_1}+\cdots+d_{i_r}=|I|-2}}
\langle\tau_0\tau_{d_I}\rangle_0 
F_l(
\begin{smallmatrix}
0&d_{j_1}&\cdots&d_{j_s}\\ 
f_I&f_{j_1}&\cdots&f_{j_s}
\end{smallmatrix}; 
d_0
)+\\+
\sum_{\substack{1\notin I\subset\underline{n},\\ d_0+d_{i_1}+\cdots+d_{i_r}=|I|-1}} 
\langle\tau_{d_0}\tau_0\tau_{d_I}\rangle_0 
f_I F_l(
\begin{smallmatrix}
d_{j_1}&\cdots&d_{j_s}\\ 
f_{j_1}&\cdots&f_{j_s}
\end{smallmatrix}; 
0
).
\end{multline} 
\end{lemma}

\begin{proof}
To prove this lemma, we shall examine the sum
\begin{equation}
F_l(
\begin{smallmatrix}
d_1+1&d_2&\cdots&d_n\\ 
f_1&f_2&\cdots&f_n
\end{smallmatrix}; 
d_0
)+ 
F_l(
\begin{smallmatrix}
d_1&d_2&\cdots&d_n\\ 
f_1&f_2&\cdots&f_n
\end{smallmatrix}; 
d_0+1
),
\end{equation}
having in mind the symmetrised version of the genus~$0$ topological recursion relation~\eqref{TRR-0-sym}.

Let us split that sum into several parts, according to Formula \eqref{identity-d} and to the position on the label~$1$ in the summands of that formula:
\begin{equation}\label{lhs-part-1}
\left(\langle\tau_{d_0+l-1}\tau_{d_1+1}\cdots\tau_{d_n}\rangle_0+\langle\tau_{d_0+l}\tau_{d_1}\cdots\tau_{d_n}\rangle_0\right)D_l(f_1f_2\cdots f_n),
\end{equation}

\begin{equation}\label{lhs-part-2a}
(-1)^l\left(\langle\tau_{d_0}\tau_{d_1+l}\cdots\tau_{d_m}\cdots\tau_{d_n}\rangle_0+\langle\tau_{d_0+1}\tau_{d_1+l-1}\cdots\tau_{d_m}\cdots\tau_{d_n}\rangle_0\right)D_l(f_1)f_2\cdots f_n,
\end{equation}

\begin{multline}\label{lhs-part-2b}
(-1)^l\sum_{m=2}^n \left(\langle\tau_{d_0}\tau_{d_1+1}\cdots\tau_{d_m+l-1}\cdots\tau_{d_n}\rangle_0+\right.\\\left.+\langle\tau_{d_0+1}\tau_{d_1}\cdots\tau_{d_m+l-1}\cdots\tau_{d_n}\rangle_0\right)f_1\cdots D_l(f_m)\cdots f_n,
\end{multline}

\begin{equation}\label{lhs-part-3a}
\sum_{\substack{\{1\}\sqcup I\sqcup J=\underline{n}, |I|\ge1,\\ i+j=l-2}}(-1)^{j+1}\left(\langle\tau_{i}\tau_{d_1+1}\tau_{d_I}\rangle_0
\langle\tau_{d_0}\tau_{j}\tau_{d_J}\rangle_0+\langle\tau_{i}\tau_{d_1}\tau_{d_I}\rangle_0
\langle\tau_{d_0+1}\tau_{j}\tau_{d_J}\rangle_0\right) D_l(f_1f_I) f_J,
\end{equation}

and

\begin{equation}\label{lhs-part-3b}
\sum_{\substack{\{1\}\sqcup I\sqcup J=\underline{n}, |I|\ge2,\\ i+j=l-2}}(-1)^{j+1}\left(\langle\tau_{i}\tau_{d_I}\rangle_0
\langle\tau_{d_0}\tau_{j}\tau_{d_1+1}\tau_{d_J}\rangle_0+
\langle\tau_{i}\tau_{d_I}\rangle_0 \langle\tau_{d_0+1}\tau_{j}\tau_{d_1}\tau_{d_J}\rangle_0\right) D_l(f_I) f_1f_J.
\end{equation}

Our goal now is to rewrite these sums so as to obtain directly the contributions from
\begin{equation}\label{rhs-part-1}
\sum_{\substack{1\in I\subset\underline{n},\\ d_{i_1}+\cdots+d_{i_r}=|I|-2}}
\langle\tau_0\tau_{d_I}\rangle_0 
F_l(
\begin{smallmatrix}
0&d_{j_1}&\cdots&d_{j_s}\\ 
f_I&f_{j_1}&\cdots&f_{j_s}
\end{smallmatrix}; 
d_0
) 
\end{equation}

and from

\begin{equation}\label{rhs-part-2}
\sum_{\substack{1\notin I\subset\underline{n},\\ d_0+d_{i_1}+\cdots+d_{i_r}=|I|-1}} 
\langle\tau_{d_0}\tau_0\tau_{d_I}\rangle_0 
f_I F_l(
\begin{smallmatrix}
d_{j_1}&\cdots&d_{j_s}\\ 
f_{j_1}&\cdots&f_{j_s}
\end{smallmatrix}; 
0
). 
\end{equation}

There are several cases where we can apply the topological recursion relation \eqref{TRR-0-sym} directly. Applying it to \eqref{lhs-part-1}, \eqref{lhs-part-2a}, \eqref{lhs-part-2b}, and \eqref{lhs-part-3b}, we recognise the respective summands from \eqref{rhs-part-1} and \eqref{rhs-part-2}. The only part of our sum where \eqref{TRR-0-sym} cannot be applied directly is \eqref{lhs-part-3a}, where $d_0$ and $d_1$ appear in different correlators. In the remaining part of the proof, we shall concentrate on analysing that part.

Let us rewrite \eqref{lhs-part-3a} using the topological recursion relation \eqref{TRR-0-sym} in the forms
\begin{equation}\label{apply-recursion-1}
\langle\tau_{i}\tau_{d_1+1}\tau_{d_I}\rangle_0=-\langle\tau_{i+1}\tau_{d_1}\tau_{d_{I}}\rangle_0+\sum_{K\sqcup L=I}\langle\tau_{i}\tau_0\tau_{d_L}\rangle_0\langle\tau_{0}\tau_{d_1}\tau_{d_K}\rangle_0
\end{equation}
and
\begin{equation}\label{apply-recursion-2}
\langle\tau_{d_0+1}\tau_{j}\tau_{d_J}\rangle_0=-\langle\tau_{d_0}\tau_{j+1}\tau_{d_{J}}\rangle_0+\sum_{K\sqcup L=J}\langle\tau_{d_0}\tau_0\tau_{d_L}\rangle_0\langle\tau_{0}\tau_{j}\tau_{d_K}\rangle_0.
\end{equation}
The sums on the right hand sides of \eqref{apply-recursion-1} and \eqref{apply-recursion-2} give most of the missing summands in \eqref{rhs-part-1} and \eqref{rhs-part-2} respectively. The terms $-\langle\tau_{i+1}\tau_{d_1}\tau_{d_{I}}\rangle_0$ and $-\langle\tau_{d_0}\tau_{j+1}\tau_{d_{J}}\rangle_0$, on the other hand, give rise to the terms
\begin{equation}
(-1)^j\langle\tau_{i+1}\tau_{d_1}\tau_{d_{I}}\rangle_0\langle\tau_{d_0}\tau_{j}\tau_{d_{J}}\rangle_0 D_l(f_1f_I)f_J
\end{equation}
and
\begin{equation}
(-1)^j\langle\tau_{i}\tau_{d_1}\tau_{d_{I}}\rangle_0\langle\tau_{d_0}\tau_{j+1}\tau_{d_{J}}\rangle_0 D_l(f_1f_I)f_J
\end{equation}
almost all of which can be grouped into pairs appearing with opposite signs and cancelling one another. The two terms that remain are the first and the last one,
\begin{equation}
(-1)^l\langle\tau_{0}\tau_{d_1}\tau_{d_{I}}\rangle_0\langle\tau_{d_0}\tau_{l-1}\tau_{d_{J}}\rangle_0 D_l(f_1f_I)f_J,
\end{equation}
and
\begin{equation}
\langle\tau_{l-1}\tau_{d_1}\tau_{d_{I}}\rangle_0\langle\tau_{d_0}\tau_0\tau_{d_{J}}\rangle_0 D_l(f_1f_I)f_J
\end{equation}
which are precisely the missing terms from \eqref{rhs-part-1} and \eqref{rhs-part-2} respectively.
\end{proof}

Using this lemma, we complete proof by induction. Indeed, \eqref{recursion-on-identities} allows us to make one of the $d_i$, for $i>0$, smaller at the cost of either increasing $d_0$ (keeping the sum $d_0+d_1+\cdots+d_n$ fixed) or decreasing $n$. Since we know, for a fixed~$l$, that the identities hold for small $n$ and for $d_1=d_2=\cdots=d_n=0$ for all~$n$, this is enough to complete the third (and the last) step of the proof by induction.
\end{proof}

\section{Stabilisers in genera~0 and~1}\label{sec:wheeled-ho-bv-givental}

In this section, we shall explore the property of an operator to preserve a CohFT a bit further, examining the conditions imposed by considering a TFT in genera 0 and 1. As we remarked earlier, this corresponds to dealing with wheeled operads. Thus, we assume we are given the classes
\begin{gather}
\alpha_n=\alpha_{0,n+1}\in H^0(\overline{\calM}_{0,n+1},\mathbb{C})\otimes\cEnd_V(n),\\  
\beta_n=\alpha_{1,n}\in H^0(\overline{\calM}_{1,n},\mathbb{C})\otimes\cEnd_V(n).
\end{gather}

Let us define a wheeled version of commutative $\BV_\infty$-algebras; the reasoning behind this definition can be recovered from homotopical calculations in Appendix~A. We want to emphasize that our definition of wheeled $\BV$-algebras includes the Getzler relation (also known as ``$1/12$-axiom'' for cyclic dg $\BV$-algebras); Getzler was probably the first to have understood the significance of that relation equation in the wheeled version of the operad~$\BV$.

\begin{definition}\label{wheeled-comm-homotopy-bv}[Wheeled commutative $\BV_\infty$-algebra]
A graded commutative algebra $V$ with a trace is said to be a wheeled commutative $\BV_\infty$-algebra if it is equipped with a collection of operators $D_l$, $l\ge1$, making it a commutative $\BV_\infty$-algebra, for which, in addition, the Getzler relation
\begin{equation}
\frac{1}{12}\str(D_2(f_1)\cdot(-))=\str(D_2(f_1\cdot(-)))\label{one-twelve}
\end{equation}
holds, and the operators $D_k$ for $k\ge3$ are strongly compatible with the trace.
\end{definition}

We are now ready to formulate the main result of this section. 

\begin{theorem}\label{wheeled-comm-bv-infty}
Let $V$ be a graded vector space. The Givental action of the element $$\mathbf{D}=\sum_{l\ge1}D_lz^{l-1}\in\End(V)[[z]]$$ preserves a given $V$-valued TFT in genera $0$ and~$1$ if and only if each linear operator $D_k$ is a differential operator of order at most $k$ strongly compatible with the trace on the commutative algebra~$V$.
\end{theorem}

\begin{corollary}\label{wheeled-comm-bv-infty-cor}
Let $V$ be a chain multicomplex with structure maps $D_l$, $l\ge 1$. The Givental action of the element $\mathbf{D}=\sum_{l\ge1}D_lz^{l-1}\in\End(V)[[z]]$ preserves a given $V$-valued TFT in genera $0$ and~$1$ if and only if the commutative algebra~$V$ with a trace equipped with the operators $D_i$ is a wheeled commutative~$\BV_\infty$-algebra.
\end{corollary}

\begin{proof}
 Let us first outline the main new feature of this proof, in comparison to the genus~$0$ case. Following the same strategy as for the genus~$0$ result, we shall be able to prove that the wheeled commutative $\BV_\infty$-algebra constraints are satisfied if and only if the Givental action of the element $\mathbf{D}=\sum_{l\ge1}D_lz^{l-1}\in\End(V)[[z]]$ preserves the correlators of the corresponding $V$-valued TFT in genera $0$ and~$1$. In the genus~$0$ case, the whole cohomology ring of $\overline{\calM}_{0,n}$ coincides with the tautological ring, so preserving a TFT in genus~$0$ is equivalent to preserving all its correlators. For the genus $1$, this is not the case, and extra arguments are needed to conclude that preserving the correlators of a TFT is sufficient for a Givental Lie algebra element to preserve that TFT. First of all, a standard argument (see e.g. \cite{FSZ10}) shows that the intersection of the deformed classes defining the TFT with any tautological class is equal to zero. It remains to 
note 
that formula \eqref{Givental-general} for the Givental Lie algebra action implies that the deformations are given by tautological classes, and the Gorenstein conjecture for the moduli spaces $\overline{\calM}_{1,n}$, $n\geq 1$ \cite{HL,Pan}, proved recently in~\cite{Pet12}, ensures that the tautological rings in genus~$1$ have perfect pairings, that is, if the intersection of a given tautological class $\xi$ with any other tautological class is equal to zero, then $\xi=0$. Therefore, it is sufficient to mimic the genus~$0$ approach, and to work with correlators only.

Because the TFT data consists of degree~$0$ classes, we conclude, as in the genus~$0$ case, that the Givental action of $\mathbf{D}$ preserves a given TFT if and only if its components preserve it, so we examine the vanishing conditions
\begin{equation}\label{preserving-CohFT-individually-genus-1}
\{\alpha'(l),\beta'(l)\}:=\widehat{D_l}.\{\alpha,\beta\}=0, \quad \text{ for }l\ge 1\ .
\end{equation}

Because of Theorem \ref{comm-bv-infty}, we already know that the genus~$0$ conditions mean that we obtain a commutative $\BV_\infty$-algebra structure. The genus one part of the proof begins similarly to the proof of Theorem~\ref{comm-bv-infty}, but in fact follows largely from what we already know in genus zero. By Proposition~\ref{order-wheel}, the genus zero conditions guarantee that ``in the presence of the trace, $D_l$ is of order at most~$l-1$''.

Similarly to the genus zero case, the vanishing conditions for the classes $\beta'_n(l)$ are satisfied trivially for $n<l-1$ for cohomological degree reasons, since $\dim_\mathbb{C}\overline{\calM}_{1,n}=n$. For $n\ge l-1$, we shall express our conditions via correlators. For every sequence of nonnegative integers $d_1,\ldots,d_n$ with 
$l-1+d_1+\cdots+d_n=n$, 
we compute the intersection of $\beta'_n(l)$ with $\prod_{m=1}^n\psi_m^{d_m}$, thus obtaining the condition on an element of $\cEnd_V(n)$ instead. To take care of signs in subsequent computations, we introduce the operations with $n$ inputs and no outputs
\begin{equation}
G_l(\begin{smallmatrix}d_1&d_2&\cdots&d_n\\ -&-&\cdots&-\end{smallmatrix}):=(-1)^l\int_{\overline{\mathcal{M}}_{1,n}} \beta'_n(l)\cdot \prod_{m=1}^n\psi_m^{d_m}.
\end{equation}
Substituting elements $f_1,\ldots,f_n\in V$ in such an operation, we get the system of identities
\begin{equation}\label{numerical-condition-d-genus-1}
G_l(\begin{smallmatrix}d_1&d_2&\cdots&d_n\\ f_1&f_2&\cdots&f_n\end{smallmatrix})=0. 
\end{equation}
Moreover, Formula \eqref{Givental-action-genus-1} implies that
\begin{multline}\label{identity-d-genus-1}
G_l(\begin{smallmatrix}d_1&d_2&\cdots&d_n\\ f_1&f_2&\cdots&f_n\end{smallmatrix})=
(-1)^l\sum_{m=1}^n \langle\tau_{d_1}\cdots\tau_{d_m+l-1}\cdots\tau_{d_n}\rangle_1\str(f_1\cdots D_l(f_m)\cdots f_n\cdot(-))+\\+
\sum_{\substack{I\sqcup J=\underline{n}, |I|\ge2,\\ i+j=l-2}}(-1)^{j+1}\langle\tau_{i}\tau_{d_I}\rangle_0
\langle\tau_{j}\tau_{d_J}\rangle_1 \str(D_l(f_I) f_J\cdot(-))+\\+
\frac12\sum_{i+j=l-2}(-1)^{j+1}\langle\tau_i\tau_{d_1}\cdots\tau_{d_n}\tau_j\rangle_0 \str(D_l(f_1\ldots f_n\cdot(-))).
\end{multline}

Examining the sum
\begin{equation}
\frac12\sum_{i+j=l-2}(-1)^{j+1}\langle\tau_i\tau_{d_1}\cdots\tau_{d_n}\tau_j\rangle_0 \str(D_l(f_1\ldots f_n\cdot(-))),
\end{equation}
we see that the formula in the case $l=2$ is substantially different from the case $l>2$ (which accounts for the difference between the Getzler relation and the compatibility with the trace): 
in the former case, this sum is just one summand
\begin{equation}
-\frac12\langle\tau_0^2\tau_{d_1}\cdots\tau_{d_n}\rangle_0 \str(D_l(f_1\ldots f_n\cdot(-))),
\end{equation}
while in the latter case it is an alternating sum of binomial coefficients \eqref{correlator-numbers} and therefore vanishes. Hence we shall consider these cases separately.

Let us first consider the case $l=2$. We first examine this identity for $d_1=\cdots=d_n=0$ (this forces $n=l-1$):
\begin{multline}\label{l=2}
G_2(\begin{smallmatrix}0\\ f_1\end{smallmatrix})=\langle\tau_1\rangle_1\str(D_2(f_1)\cdot(-))-\frac12\langle\tau_0^3\rangle_0\str(D_2(f_1\cdot(-)))=\\=\frac{1}{24}\str(D_2(f_1)\cdot(-))-\frac12\str(D_2(f_1\cdot(-))),
\end{multline}
so we recover the Getzler $1/12$-relation.

Now we proceed with arbitrary $d_1,\ldots,d_n$. Formula~\eqref{identity-d-genus-1} in this case becomes
\begin{multline}\label{l=2-1}
G_2(\begin{smallmatrix}d_1&d_2&\cdots&d_n\\ f_1&f_2&\cdots&f_n\end{smallmatrix})=
\sum_{m=1}^n \langle\tau_{d_1}\cdots\tau_{d_m+1}\cdots\tau_{d_n}\rangle_1\str(f_1\cdots D_2(f_m)\cdots f_n\cdot(-))-\\-
\sum_{I\sqcup J=\underline{n}, |I|\ge2}\langle\tau_{0}\tau_{d_I}\rangle_0\langle\tau_{0}\tau_{d_J}\rangle_1 \str(D_2(f_I) f_J\cdot(-))-
\frac12\langle\tau_0^2\tau_{d_1}\cdots\tau_{d_n}\rangle_0 \str(D_2(f_1\ldots f_n\cdot(-)))
\end{multline}

Let us rewrite the $m^\text{th}$ summand in the first sum using the genus one topological recursion relation~\eqref{TRR-1} at the point~$m$, obtaining
\begin{equation}\label{l=2-2}
\langle\tau_{d_1}\cdots\tau_{d_m+1}\cdots\tau_{d_n}\rangle_1=
\left(\frac{1}{24}\langle\tau_0^2\tau_{d_1}\cdots\tau_{d_n}\rangle_0+\sum_{I\sqcup J\sqcup\{m\}=\underline{n}}
\langle\tau_{d_I}\tau_{d_m}\tau_0\rangle_0\langle\tau_0\tau_{d_J}\rangle_1\right)
\end{equation}

If we apply the Getzler $1/12$-relation (which we already obtained examining~\eqref{l=2}) to the last term of~\eqref{l=2-1}, and group the result with the sum of the first terms of~\eqref{l=2-2}, the result cancels because of Identities~\eqref{drop-order-1} and~\eqref{drop-order-2} (applied to the operator $D_2$ of order~$2$, those equations imply that it becomes an operator of order~$1$ in the presence of the trace). Grouping together the remaining terms with the same $J$, we see that each of the groups vanishes because of Identities \eqref{drop-order-1} and~\eqref{drop-order-2}.

Let us now consider the case $l>2$. For $d_1=d_2=\cdots=d_n=0$ (this forces $n=l-1$), we have 
\begin{multline}\label{l>2}
G_l(\begin{smallmatrix}0&0&\cdots&0\\ f_1&f_2&\cdots&f_{l-1}\end{smallmatrix})=
(-1)^l\sum_{m=1}^n \langle\tau_{l-1}\tau_0^{l-2}\rangle_1\str(f_1\cdots D_l(f_m)\cdots f_{l-1}\cdot(-))+\\+
\sum_{\substack{I\sqcup J=\underline{l-1}}}(-1)^{|J|+2}\langle\tau_{|I|-2}\tau_0^{|I|}\rangle_0
\langle\tau_{|J|+1}\tau_0^{|J|}\rangle_1 \str(D_l(f_I) f_J\cdot(-))=\phantom{aaaaaaaaaa}\\\phantom{aaa}=
\sum_{m=1}^n \frac{(-1)^l}{24}\str(D_l(f_m)f_{\underline{l-1}\setminus\{m\}}\cdot(-))+\sum_{\substack{I\sqcup J=\underline{l-1}, \\ |I|\ge2}}\frac{(-1)^{|J|+2}}{24} \str(D_l(f_I) f_J\cdot(-))=\\=
\frac{1}{24}\str(\langle f_1,\ldots,f_{l-1}\rangle_{l-1}^{D_l}\cdot(-)),
\end{multline}
and we obtain, up to a factor $1/24$, Identity \eqref{strong-compat-1}.

Let us show that all other identities $G_l(\begin{smallmatrix}d_1&d_2&\cdots&d_n\\ f_1&f_2&\cdots&f_n\end{smallmatrix})=0$ follow from Identities~\eqref{strong-compat-1} and~\eqref{strong-compat-2}. For that, we shall prove the following lemma.

\begin{lemma}\label{recursion-on-identities-genus-1}
We have 
\begin{multline}\label{recursion-formula-genus-1}
G_l(\begin{smallmatrix}d_1&d_2&\cdots&d_n\\ f_1&f_2&\cdots&f_n\end{smallmatrix})=\\=
-\sum_{\substack{I\sqcup J=\underline{n},\\ d_{i_1}+\cdots+d_{i_r}=|I|+1}}\langle\tau_0\tau_{d_{i_1}}\cdots\tau_{d_{i_r}}\rangle_1 \str\left(f_I E_{l-1}\left(\begin{smallmatrix}d_{j_1}&d_{j_2}&\cdots&d_{j_s}\\ f_{j_1}&f_{j_2}&\cdots&f_{j_s}\end{smallmatrix};0\right).(-)\right)+\\+
\frac{1}{24}\sum_{\substack{I\sqcup J=\underline{n},\\ d_{i_1}+\cdots+d_{i_r}=|I|}}\langle\tau_0^3\tau_{d_I}\rangle_0 \str\left(f_I E_{l-2}\left(\begin{smallmatrix}d_{j_1}&d_{j_2}&\cdots&d_{j_s}\\ f_{j_1}&f_{j_2}&\cdots&f_{j_s}\end{smallmatrix};0\right).(-)\right).
\end{multline}
Here $E_{l-1}$ and $E_{l-2}$ denote the expressions analogous to $F_{l-1}$ and $F_{l-2}$ respectively (as defined by Formula~\eqref{identity-d}), but based on the operator $D_l$, not $D_{l-1}$ and $D_{l-2}$.
\end{lemma}

\begin{proof}
Let us split \eqref{identity-d-genus-1} into separate parts, which we shall then treat individually. In the first sum, we apply to the $m^\text{th}$ summand the genus one topological recursion relation \eqref{TRR-1} at point~$m$, obtaining
\begin{multline}
\langle\tau_{d_1}\cdots\tau_{d_m+l-1}\cdots\tau_{d_n}\rangle_1=\left(\frac{1}{24}\langle\tau_{d_1}\cdots\tau_{d_m+l-2}\cdots\tau_{d_n}\tau_0^2\rangle_0+\right.\\+\left.
\sum_{A\sqcup B\sqcup\{m\}=\underline{n}}
\langle\tau_{d_A}\tau_{d_m+l-2}\tau_0\rangle_0\langle\tau_0\tau_{d_B}\rangle_1\right).
\end{multline}
We shall leave the sum unchanged, and only transform the first term, where we apply the topological recursion relation \eqref{TRR-0} at the points $m,n+1,n+2$ to get 
\begin{equation}
\frac{1}{24}\langle\tau_{d_1}\cdots\tau_{d_m+l-2}\cdots\tau_{d_n}\tau_0^2\rangle_0=\frac{1}{24}\sum_{K\sqcup L\sqcup\{m\}=\underline{n}}
\langle\tau_{d_K}\tau_{d_m+l-3}\tau_0\rangle_0\langle\tau_0^3\tau_{d_L}\rangle_0.
\end{equation}

The second sum in $\eqref{identity-d-genus-1}$ has the terms with $j=0$ and $j\ge1$; we consider them separately. The terms with $j=0$ have the coefficient
\begin{equation}\label{j-equal-to-0}
-\sum_{I\sqcup J=\underline{n}, |I|\ge2}\langle\tau_{l-2}\tau_{d_I}\rangle_0\langle\tau_0\tau_{d_J}\rangle_1=-\sum_{A\sqcup B=\underline{n}, |A|\ge2}\langle\tau_{l-2}\tau_{d_A}\rangle_0\langle\tau_0\tau_{d_B}\rangle_1, 
\end{equation}
which we shall keep intact, having only renamed the summation variables, while the terms with $j\ge1$ are being rewritten using the genus one topological recursion relation \eqref{TRR-1} at the point with label~$j$, producing the coefficient
\begin{equation}\label{j-ge-1}
(-1)^{j+1}\langle\tau_{i}\tau_{d_I}\rangle_0\left(\frac{1}{24}\langle\tau_{j-1}\tau_{d_J}\tau_0^2\rangle_0+ \sum_{A\sqcup B=J}\langle\tau_{d_A}\tau_{j-1}\tau_0\rangle_0\langle\tau_0\tau_{d_B}\rangle_1\right).
\end{equation}
In the latter formula, we do not transform the sum, but rewrite the first term (for every $j\ge2$), where we apply the topological recursion relation \eqref{TRR-0} at the first point and the last two points, replacing it by
\begin{equation}\label{j-ge-2}
\frac{(-1)^{j+1}}{24}\langle\tau_{i}\tau_{d_I}\rangle_0\left(
\sum_{K\sqcup L=J}
\langle\tau_{d_K}\tau_{j-2}\tau_0\rangle_0\langle\tau_0^3\tau_{d_L}\rangle_0
\right).
\end{equation}
For $j=1$, this rewriting cannot be performed, and we keep the corresponding term, just renaming the summation variables: 
\begin{equation}\label{j-equal-to-1}
\frac{1}{24}\langle\tau_{l-3}\tau_{d_I}\rangle_0\langle\tau_0^3\tau_{d_J}\rangle_0=\frac{1}{24}\langle\tau_{l-3}\tau_{d_K}\rangle_0\langle\tau_0^3\tau_{d_L}\rangle_0.  
\end{equation}
Now, examining the formulae above, we observe that they can be joined into two groups, one consisting of the terms having the common numeric coefficient $\frac{1}{24}\langle\tau_0^3\tau_{d_L}\rangle_0$, and the other one having the common numeric
coefficient $\langle\tau_0\tau_{d_B}\rangle_1$. Examining these groups, we obtain precisely Formula~\eqref{recursion-formula-genus-1}.
\end{proof}

Using this lemma, we complete the proof of our theorem, observing that since we already know that in the presence of the wheel the order of differential operators drops by two, the genus zero result guarantees that all further identities are satisfied.
\end{proof}

\appendix
\section{Differential operator conditions and (wheeled) homotopy BV-algebras}

In this appendix, we shall explain why it is natural to expect the differential operator conditions in the homotopy $\BV$ context.

\smallskip

The definition of \cite{Kra00} is beautiful but somewhat \emph{ad hoc}; there, the story starts from an assumption that the operator~$\sum_{l}D_l z^{l-1}$ (defining the structure of a chain multicomplex on~$V$, or, more informally, ``a resolution of the relation $\Delta^2=0$'') has homogeneous components that are differential operators of finite orders, which generally does not have to be true. The paper \cite{GTV09}, where a free resolution of the operad $\BV$ is obtained, gives a conceptual explanation: it is shown there that a homotopy $\BV$-algebra in which all higher homotopical operations except for the unary ones vanish is a commutative $\BV_\infty$-algebra in the sense of \cite{Kra00}.

In the wheeled case, a free resolution for $\BV$ is not readily available, so one has to approach this problem from a different angle. The explanation below is along the lines of the Koszul--Tate approach to minimal models \cite{Koszul50,Tate57}: to construct a minimal model, we resolve cycles step by step, adding higher and higher homotopies killing the appearing new cycles. It turns out that if we assume that all higher homotopies of arity two and more vanish, the differential operator conditions for operators $D_l$ acting on homotopy $\BV$-algebras appear naturally. This informal statement is formulated precisely and proved in the remaining part of this section.

\smallskip

Recall that the operad $\BV$ is made up of the algebra of dual numbers 
 $$
\D:=T(\Delta)/(\Delta^2),
 $$ 
considered as an operad concentrated in arity $1$, and the operad $\Com$ for commutative algebras as follows. The operad $\BV$ is equal to the quotient of the coproduct (the free product) of the operad $\D$ and  the operad $\Com$ by the relation saying that $\Delta$ is an order $2$ operator:
\begin{equation}
\BV=\frac{\D\vee \Com}{(\Delta\ \text{order}\ 2)} \ . 
\end{equation}
The cobar construction $\Omega \D^{\ac} \xrightarrow{\sim} \D$ on the Koszul dual coalgebra $\D^{\ac}$ of $\D$ is a resolution of $\D$. Its underlying algebra is the free algebra on a sequence of elements, also denoted $\{D_l\}_{l\ge 2}$ by a slight abuse of notation. Its differential is given by Formula~\eqref{resolving-delta}.

\begin{theorem}
The differential of $\Omega \D^{\ac}$ induces a well defined differential on the quotient operad 
\begin{equation}
\frac{\Omega \D^{\ac}\vee \Com}{(D_l\ \emph{order}\ l, \emph{for}\ l\ge2)}. 
\end{equation}
If a quotient dg operad 
\begin{equation}
\frac{\Omega \D^{\ac}\vee \Com}{(\emph{R})}
\end{equation}
is quasi-isomorphic to the operad $\BV$, then the space of relations $R$ contains the order $l$ relations for the operators $D_l$.
\end{theorem}

\begin{proof}
We begin with the second part of the statement.  
From the definition of $\BV$-algebras, $D_2$ corresponds to $\Delta$ and therefore, it must be a second order differential operator in any the resolutions of $\BV$. Let us prove that $D_l$ must be  a differential operator of order at most~$l$ by induction, using $l=2$ as the basis. Let us make the inductive step; by the inductive hypothesis, we may assume that $l\ge2$ and $D_i$ for $i=2,\ldots,l$ is a differential operator of order~$i$.  Note that because of the homological degrees Equations \eqref{resolving-delta} can be written in the form
\begin{equation}
\sum_{i+j=n}[D_i,D_j]=0, 
\end{equation}
so
\begin{equation}
[D_1,D_{l+1}]=-\frac12\sum_{i=2}^{l}[D_i,D_{l+2-i}].
\end{equation}
Let us consider, as in Section \ref{com-bv-infty-via-givental}, the composite of operations $F_{l+1}(\begin{smallmatrix}0&0&\cdots&0\\ -&-&\cdots&-\end{smallmatrix};0)$ but viewed in the resolution 
$\frac{\Omega \D^{\ac}\vee \Com}{(\emph{R})}\xrightarrow{\sim} \BV$
 of~$\BV$ this time. It vanishes exactly when $D_{l+1}$ is a differential operator of order~$l+1$. We compute its differential in the resolution and  we denote the result by~$R_{l+1}$. That amounts to replacing~$D_{l+1}$ everywhere in that composite by $-\frac12\sum_{i=2}^{l}[D_i,D_{l+2-i}]$. In the absence of homotopies of arity $2$ and higher up, all the differential operators in that sum are of order at most~$l+1$ because of Property~\eqref{filtered-Lie}, so $R_{l+1}=0$, and the condition for $D_{l+1}$ to be a differential operator of order~$l+1$ is a cocycle. (This proves the first statement.) Hence it has to be resolved using a homotopy of arity at least~$2$, and that homotopy vanishes in the resolution by the assumption. 
\end{proof}

In the genera $0$ and $1$ case, we work with the wheeled analogues of the aforementioned operads. Let $\BV^\circlearrowright$ denote the wheeled operad which encodes $\BV$-algebras equipped with a trace satisfying the Getzler $1/12$-relation. The wheeled operads  $\Com^\circlearrowright$ and $\Omega^\circlearrowright \D^{\ac}$ are the wheelification of the corresponding operads, that is wheeled versions without any more relations.

\begin{theorem}\label{thm:WheeledDiffOp}
The differential of $\Omega \D^{\ac}$ induces a well defined differential on the quotient wheeled operad 
\begin{equation}
\frac{\Omega^\circlearrowright \D^{\ac}  \vee \Com^\circlearrowright}{(\emph{Getzler 1/12-relation};\ D_l\ \emph{order}\ l \ \&\  \emph{strongly compatible with the trace}, \ \emph{for}\ l\ge2)}.  
\end{equation}
If a quotient wheeled dg operad 
\begin{equation}
\frac{\Omega^\circlearrowright \D^{\ac}\vee \Com}{(\emph{R})} 
\end{equation}
is quasi-isomorphic to the wheeled operad $\BV^\circlearrowright$, then the space of relations $R$ contains the 
Getzler $1/12$-relation, the order $l$ relations for the operators $D_l$ and the strong compatibility with the trace. 
\end{theorem}

\begin{proof}
For wheeled operads, the same strategy works, but instead of Property~\eqref{filtered-Lie}, we shall use the notion of compatibility with the trace. Proposition~\ref{lie-subalgebra} shows that to mimic the proof in the operadic case, we need to deal with the basis of the induction carefully (since in wheeled commutative $\BV_\infty$-algebras the operator $D_2$ is not strongly compatible with the trace), and deal with the case $[D_2,D_l]$ which is the only term not covered by the induction assumption. These terms are handled by the following lemma.

\begin{lemma}\label{comm-compat}
Let $D_2$ be  a differential operator of odd degree and of order at most $2$ satisfying the Getzler $1/12$-relation, and let $D_l$ be a differential operator of odd degree and of order at most $l$, which is strongly compatible with the trace. Then both the operators $[D_2,D_2]$ and $[D_2,D_l]$ are strongly compatible with the trace, as operators of order at most $3$ and at most $l+1$ respectively. 
\end{lemma}

\begin{proof}
We shall prove these statements simultaneously, mimicking the proof of Proposition~\ref{lie-subalgebra}. Namely, we rewrite  $\str(f\langle f_1,\ldots,f_{k}\rangle_{k}^{[D_2,D_l]}\cdot(-))$ as
\begin{multline}
\sum_{\substack{I\sqcup J = \underline{l},\\ |I|=r\ge1}} 
\str\left(f\langle \langle f_{i_1},\ldots,f_{i_r}\rangle_r^{D_l},f_{j_1},\ldots,f_{j_{l-r}}\rangle_{l+1-r}^{D_2}\cdot(-)\right)+\\+
\str\left(f\langle \langle f_{i_1},\ldots,f_{i_r}\rangle_r^{D_2},f_{j_1},\ldots,f_{j_{k-r}}\rangle_{l+1-r}^{D_l}\cdot(-)\right).
\end{multline}
The first of the summands in each of the terms of this sum vanishes for $l+1-r\ge 2$ (that is for $l-1\ge r$) by Proposition~\ref{order-wheel} because $D_2$ is of order at most~$2$. The second one is equal to the first one for $l=2$, and for $l>2$ vanishes for $r\ge 3$ by Proposition~\ref{order-wheel} because $D_2$ is of order at most $2$ and for $l+1-r\ge l-1$ (that is for $2\ge r$) because $D_l$ is of order at most~$l$ and is strongly compatible with the trace. Therefore the only term that does not vanish \emph{a priori} is
\begin{equation}\label{remaining-str}
\str\left(f\langle \langle f_1,\ldots,f_l\rangle_l^{D_l}\rangle_{1}^{D_2}\cdot(-)\right).
\end{equation}
Let us prove that this term vanishes. Both the operator $D_2$ and the operator $D_l$ are of order at most~$l$, so we write
\begin{equation}\label{str-1}
\sum_{\varnothing\ne I\subset\underline{l+1}}(-1)^{l+1-|I|}D_l(f_I)f_{\underline{l+1}\setminus I}=0
\end{equation}
and
\begin{equation}\label{str-2}
\sum_{\varnothing\ne I\subset\underline{l+1}}(-1)^{l+1-|I|}D_2(f_I)f_{\underline{l+1}\setminus I}=0.
\end{equation}
Let us apply $D_2$ to Identity \eqref{str-1}, apply $(-1)^lD_l$ to Identity~\eqref{str-2}, add the results, and compute the traces with respect to the input $f_{l+1}$. Using the fact that traces vanish on commutators, it is possible to cancel all the terms $\str(D_2(D_l(f_I\cdot(-))f_J))$ with the terms $\str(D_l(D_2(f_I\cdot(-))f_J))$. It follows that
\begin{multline}
\str\left(D_2\left(\sum_{\varnothing\ne I\subset\underline{l}}(-1)^{l-|I|}D_l(f_I)f_{\underline{l}\setminus I}\cdot(-)\right)\right)+\\+
(-1)^l\str\left(D_l\left(\sum_{\varnothing\ne I\subset\underline{l}}(-1)^{l-|I|}D_2(f_I)f_{\underline{l}\setminus I}\cdot(-)\right)\right)=0.
\end{multline}
The latter identity can be rewritten as
\begin{equation}
\str\left(\langle\langle f_1,\ldots,f_l\rangle_l^{D_l}\cdot(-)\rangle_1^{D_2}\right)+(-1)^l\str\left(\langle\langle f_1,\ldots,f_l\rangle_l^{D_2}\cdot(-)\rangle_1^{D_l}\right)=0.
\end{equation}
For $l>2$, the second term vanishes since already $\langle f_1,\ldots,f_l\rangle_l^{D_2}$ vanishes; for $l=2$, the two terms are the same. In either case, applying the Getzler $1/12$-relation, we deduce that \eqref{remaining-str} vanishes. 
\end{proof}

Now everything is ready to finish the proof of Theorem~\ref{thm:WheeledDiffOp}. From the definition of wheeled $\BV$-algebras, $D_2$ corresponds to $\Delta$ and, hence, it should be a second order differential operator satisfying the Getzler $1/12$-relation in any resolution of $\BV^\circlearrowright$ in  the absence of higher homotopies of arity $\ge 2$. From the operadic proof, we already know that each operator $D_l$ is a differential operator of order at most~$l$. Let us prove that, for $l\ge 3$, these operators are strongly compatible with the trace. We consider the composite of the operations $G_{l}(\begin{smallmatrix}0&0&\cdots&0\\ -&-&\cdots&-\end{smallmatrix})$ and $G_{l}(\begin{smallmatrix}0&0&\cdots&1\\ -&-&\cdots&-\end{smallmatrix})$ in the 
resolution 
of the wheeled operad $\BV^\circlearrowright$. They  vanish precisely when $D_l$ is strongly compatible with the trace. We compute their differentials in the resolution and we denote the results by~$S_l$ and $T_l$ respectively. Computing the differential amounts to replacing~$D_{l}$ everywhere in that composite of operations by $-\frac12\sum_{i=2}^{l}[D_i,D_{l+1-i}]$. In the absence of homotopies of arity~$2$ and higher, all the differential operators in that sum are strongly compatible with the trace by Lemma~\ref{comm-compat}, induction and Proposition~\ref{lie-subalgebra} (and for the basis of induction $l=3$ --- by the first part of Lemma~\ref{comm-compat}), so $S_l=T_l=0$. (This proves the first statement.) Therefore the condition for $D_l$ to be strongly compatible with the trace is a cocycle; so it has to be resolved using a homotopy of arity at least~$2$. Finally that homotopy vanishes  by assumption. This completes the proof.
\end{proof}

\section{Generalised BCOV theory}

In this Appendix, we explain a possible application of Corollaries \ref{comm-bv-infty} and~\ref{wheeled-comm-bv-infty}.

\subsection{Motivation: classical BCOV theory}

BCOV theory is a way to construct cohomological field theories from a differential graded $\BV$-algebra, when the differential and the $\BV$-operator satisfy the Hodge condition. 

A dg $\BV$-algebra is made up of a mixed chain complex 
 $$
(A, d, \Delta),\quad d^2=\Delta^2=d\Delta+\Delta d=0,
 $$ 
equipped with a compatible commutative product. Let us describe various types of Hodge conditions that occur for this setup.

\begin{definition}[Hodge conditions]  $ \ $
\begin{itemize}
\item[$\diamond$]
The compatibility relation 
\begin{eqnarray}\label{eqn:dDelta}
\mathrm{Ker} \,  d \cap \mathrm{Ker}\, \Delta \cap (\mathrm{Im} \,d + \mathrm{Im} \,\Delta)=\mathrm{Im} \,d\Delta 
\end{eqnarray}
 is called the \emph{$d\Delta$-condition} \cite{DGMS}.

\item[$\diamond$] The mixed chain complex is called \emph{semi-classical} \cite{Par} if every homology class with respect to the differential $d$ has a representative in  the kernel of $\Delta$.

\item[$\diamond$]  A {\em Hodge-to-de Rham degeneration datum}  consists of a deformation retract 
\begin{eqnarray}
&\xymatrix{     *{ \quad \ \  \quad (A, d)\ } \ar@(dl,ul)[]^{h}\ \ar@<0.5ex>[r]^-{p} & *{\
(H(A),0)\ ,\quad \ \  \ \quad }  \ar@<0.5ex>[l]^-{i}}
\end{eqnarray}
 such that 
\begin{eqnarray}\label{eqn:NCHdR}
p (\Delta h)^{m-1}\Delta i=0 \ , \quad \text{for}\  m\ge 1 \ . 
\end{eqnarray} 
\end{itemize}
\end{definition}

It is easy to see that the following implications hold 
$$(d\Delta\text{-condition}) \Longrightarrow (\text{semi-classical}) \Longrightarrow (\text{Hodge-to-de Rham degeneration datum}) \ . $$

\begin{theorem}[\cite{DV}]
The underlying homology groups of a dg BV-algebra equipped with a Hodge-to-de Rham degeneration datum is endowed with a genus~$0$ cohomological field theory structure extending the induced commutative product. 
\end{theorem}

This theorem was first proved under the $d\Delta$-condition in \cite{BarKon, Man}. Explicit formulae were given in \cite{LosSha}. It was shown under the semi-classical hypothesis in \cite{Par}. It was finally proved, in this most general form, in \cite{DV}, using the Homotopy Transfer Theorem (HTT for short) for the minimal model of the operad BV: the homology groups of a dg BV-algebra can be endowed with a skeletal homotopy BV-algebra structure extending the induced commutative product. Moreover, in the presence of Hodge-to-de Rham degeneration datum, the transferred operator $\Delta$ and its higher homotopies vanish and one gets a homotopy genus~$0$ CohFT with trivial differential. Therefore, its first stratum is a genus~$0$ CohFT.

The latter arguments can be applied to a commutative $BV_\infty$-algebra.  The HTT says that the transferred operators on homology are equal to 
\begin{equation}
\sum_{\substack{l_1+l_2+\cdots+l_k-k=n,\\ l_1, \ldots, l_k\ge 2}}(-1)^{k-1} \, pD_{l_1}hD_{l_2}h\ldots hD_{l_k}i \ .
\end{equation}

\begin{proposition}
Let $A$ be a commutative $\textrm{BV}_\infty$-algebra such that its operators $\{D_l\}_{l\ge 1}$ satisfy the following condition with the underlying deformation retract on the homology groups:
\begin{eqnarray}\label{eqn:HomotopyNCHdR}
\sum_{\substack{l_1+l_2+\cdots+l_k-k=n,\\ l_1, \ldots, l_k\ge 2}}(-1)^{k-1} \, pD_{l_1}hD_{l_2}h\ldots hD_{l_k}i=0 \ .
\end{eqnarray}
In this case, the underlying homology groups are endowed with a genus~$0$ cohomological field theory structure extending the induced commutative product.
\end{proposition}

It works like that only if we want to construct a CohFT on the operadic level; in the wheeled operad case, we have to start with a finite dimensional wheeled BV-algebra with the Hodge condition as an input (that is, we have to start with a finite-dimensional BV-algebra that satisfies the Getzler relation), and in the modular operad case, we need again a finite dimensional BV-algebra with the Hodge condition, satisfying the Getzler relation, and equipped with a scalar product that is compatible with all the structure. The required algebraic formalism was developed in~\cite{BCOV,BarKon,Man,LosSha,LosShaShn,Sha09}.

The main problem of this approach in genera higher than $0$ is the requirement for the input to be finite dimensional, since the basic examples of BV-algebras with the Hodge condition~\cite{BCOV, Mer} are infinite-dimensional. In fact, we don't know of any \emph{natural} finite-dimensional example of a wheeled BV-algebra satisfying the Hodge property.

There are different ways to resolve these kinds of problems. One of the possible approaches is discussed in the recent preprint~\cite{CosLi}, where a renormalization procedure is proposed. The price for renormalization is actually the loss of the algebraic elegancy of the BCOV theory. In the vein of \cite{DV}, one could use the HTT for the associated wheeled operad or, even better, wheeled properad, as in \cite{Merkulov10}.

Another possible strategy is to look for natural examples of a structure weaker than wheeled BCOV theory, where the same simple algebraic formalism would give explicit formulae for the induced CohFT. 
Here we explain the explicit formulae for the induced CohFT at the wheeled operad level for the input recollected from Theorems 2 and 3. To this aim, we need an analogue of the Hodge condition (\ref{eqn:NCHdR}) or (\ref{eqn:HomotopyNCHdR}) but this time for wheeled commutative $\textrm{BV}_\infty$-algebras.

Of course, in genus $0$, the formulae that we obtain must be specializations of the HTT formulae in~\cite{DV}, but, using our results one can present them in a very simple way.

\subsection{The construction} We consider a wheeled commutative $\textrm{BV}_\infty$-algebra  on a vector space $V$. 
Its structure, according to Theorems \ref{comm-bv-infty} and \ref{wheeled-comm-bv-infty}, can be described by a $V$-valued TFT that consist of classes 
\begin{equation}
\alpha_n\in H^0(\overline{\calM}_{0,n+1},\mathbb{C})\otimes\Hom(V^{\otimes n},V), n\geq 3,
\end{equation}
and
\begin{equation}
\beta_n\in H^*(\overline{\calM}_{1,n},\mathbb{C})\otimes\Hom(V^{\otimes n},\mathbb{C}), n\geq 1, 
\end{equation}
and the operators $D_l$ of order at most $l$, for each $l\ge 1$, such that the series 
\begin{equation}
D(z):=\sum_{l=1}^\infty D_lz^{l-1} 
\end{equation}
satisfies $D(z)^2=0$, and $\widehat{D(z)}.\{\alpha,\beta\}=0$.

\begin{definition}[Gauge Hodge condition]
A wheeled commutative $\textrm{BV}_\infty$-algebra $V$ is said to satisfy the \emph{gauge Hodge condition} if there exists a series $A(z)=\sum_{l=1}^\infty A_lz^l$ in $\End(V)[[z]]$ such that 
\begin{equation} \label{weakhodge}
\exp(-A(z))D_1\exp(A(z)) = D(z) .
\end{equation}
\end{definition}

This definition relaxes the semi-classical Hodge condition: if we restrict ourselves to the particular case $D(z)=D_1+D_2z$, $A(z)=A_1z$, and add the requirement $A_1^2=0$, what we obtain is precisely the semi-classical Hodge condition (see for example \cite[Remark~5.4]{Sha09}). It turns out that the conditions \eqref{eqn:HomotopyNCHdR} and \eqref{weakhodge} are equivalent, see \cite{DSV12}.

\begin{theorem} The following formula gives the structure of a wheeled operadic CohFT on the cohomology $H(V,D_1)$:
\begin{equation}
\exp(\widehat{A(z)}).\{\alpha,\beta\}_{|_{H(V,D_1)}}.
\end{equation}
\end{theorem}

\begin{proof}
The proof is almost obvious. Indeed, general Givental theory implies that the collection of classes $\exp(\widehat{A(z)}).\{\alpha,\beta\}$ defines a wheeled operadic CohFT on $V$.
Equation~\eqref{weakhodge} implies, after quantisation, that 
\begin{equation}
\widehat D_1 \exp(\widehat{A(z)}) = \exp(\widehat{A(z)}) \widehat{D} . 
\end{equation}
Therefore, $\exp(\widehat{A(z)}).\{\alpha,\beta\}$ consists of $D_1$-closed cohomology classes, and therefore, it can be restricted to the $D_1$-cohomology preserving the property of being a wheeled operadic CohFT.
\end{proof}

\bibliographystyle{amsplain}

\providecommand{\bysame}{\leavevmode\hbox to3em{\hrulefill}\thinspace}
\providecommand{\MR}{\relax\ifhmode\unskip\space\fi MR }
\providecommand{\MRhref}[2]{%
  \href{http://www.ams.org/mathscinet-getitem?mr=#1}{#2}
}
\providecommand{\href}[2]{#2}

\end{document}